\documentclass[11pt,twoside,a4paper]{article}  
\usepackage{amssymb,amsfonts,amsmath,amsbsy,theorem,amscd,dina42e,verbatim,hyperref}
\usepackage[T1]{fontenc}
\usepackage{color}
\newcommand{\dd}{\text{\it \dj\hspace{1pt}}}

\newcommand{\ol}[1]{\overline{#1}}

\numberwithin{equation}{section}
\newcommand{\C}{\ensuremath{\mathbb{C}}}
\newcommand{\R}{\ensuremath{\mathbb{R}}}

\newcommand{\Rn}{\ensuremath{\mathbb{R}^n}}

\newcommand{\N}{\ensuremath{\mathbb{N}}}

\newcommand{\F}{\ensuremath{\mathcal{F}}}

\newcommand{\dist}{\operatorname{dist}}

\newcommand{\supp}{\operatorname{supp}}
\newcommand{\eps}{\ensuremath{\varepsilon}}

\newcommand{\habil}[1]{}
\newtheorem{thm}{Theorem}[section]
\newtheorem{cor}[thm]{Corollary}

\newtheorem{lem}[thm]{Lemma}

\newtheorem{defn}[thm]{Definition}

\newtheorem{prop}[thm]{Proposition}

\newtheorem{hypo}[thm]{Hypothesis}
\newtheorem{claim*}{Claim}

\theorembodyfont{\rmfamily}
\newtheorem{rem}[thm]{Remark}

\newtheorem{example}[thm]{Example}
\newenvironment{proof*}[1]{\noindent{\bf Proof
#1:}}{\hspace*{\fill}\rule{1.2ex}{1.2ex}\\ }
\newenvironment{proof}{\noindent{\bf
Proof:\,}}{\hspace*{\fill}\rule{1.2ex}{1.2ex}\\ }

\newcommand{\crp}{\overline{\mathbb R}_+}

\newcommand{\rn}{{\mathbb R}^n}
\newcommand{\rnp}{{\mathbb R}^n_+}

\newcommand{\rnpm}{\mathbb R^n_\pm}
\newcommand{\crnp}{\overline{\mathbb R}^n_+}

\newcommand{\crnpm}{\overline{\mathbb R}^n_\pm}
\newcommand{\comega}{\overline\Omega }

\newcommand{\ang}[1]{\langle {#1} \rangle}
\newcommand{\Op}{\operatorname{OP}}

\newcommand{\simto}{\overset\sim\rightarrow}

\newcommand{\rp}{ \mathbb R_+}

\begin{document}
\begin{titlepage}
\title{Maximal $L_p$-regularity for $x$-dependent fractional heat equations
  with Dirichlet conditions} 
\author{Helmut~Abels\footnote{Fakult\"at f\"ur Mathematik,  
Universit\"at Regensburg,  
93040 Regensburg, Germany, E-mail {\tt helmut.abels@ur.de}}~ and Gerd~Grubb\footnote{Department of Mathematical Sciences,
Copenhagen University,
Universitetsparken 5,
 DK-2100 Copenhagen, Denmark.
E-mail {\tt grubb@math.ku.dk} 
}
}
\end{titlepage}
\maketitle

\begin{abstract}
We prove optimal regularity results in $L_p$-based function spaces in
space and time for a large class of linear parabolic equations with a
nonlocal elliptic operator in bounded domains with limited
smoothness. Here the nonlocal operator is given by a strongly elliptic
and even pseudodifferential operator $P$ of order $2a$ ($0<a<1$) with
nonsmooth $x$-dependent coefficients.
 {This includes the prominent case of the fractional Laplacian
  $(-\Delta )^a$, as well as elliptic operators $(-\nabla \cdot
  A(x)\nabla +b(x))^a$}. The proofs are based on general
results on maximal $L_p$-regularity and its relation to
$\mathcal{R}$-boundedness of the resolvent of the associated
(elliptic) operator. Finally, we apply these results to show existence
of strong solutions locally in time for a class of nonlinear nonlocal
parabolic equations, which include a fractional nonlinear diffusion
equation and a fractional porous medium equation after a
transformation. The nonlinear results are new in the case of domains with boundary; the linear results are so when $P$ is $x$-dependent
  nonsymmetric.
 \end{abstract}

 \noindent
 {\bf Key words:} Fractional Laplacian; even pseudodifferential
 operator;
  Dirichlet problem;
  nonsmooth coefficients;  maximal regularity;
  nonlinear nonlocal parabolic equations; fractional porous medium equation.  \\
 {\bf MSC (2020):}  Primary: 35S15, 35R11, Secondary: 35K61, 35S16, 47G30, 60G52

\section{Introduction}

 {The present paper studies the heat equation for  a nonlocal
operator $P$  of order $2a\in(0,2)$ (strongly elliptic and even),
\begin{equation}  \label{eq:1.1} 
\begin{split}
\partial_tu+Pu&=f\quad \text{ on }\Omega \times I ,\ I=\,(0,T)\,,\\
u&=0\quad \text{ on }(\mathbb R^n\setminus\Omega )\times I, \\
u|_{t=0}&=0.    
\end{split}  
\end{equation}

Linear operators $P$ of fractional order, such as the fractional Laplacian
$(-\Delta )^a$ and its generalizations, have been much in
focus in recent years, both in Analysis and in Probability and Finance. In contrast to
differential operators (always of integer order) they are
{\em nonlocal} (do not preserve the support of a function), which makes
them more difficult to handle. There are  
generally two types of definitions that are used. One is the
definition as a {\em singular integral operator}
\begin{equation}
  Pu(x)=PV\int_{\rn}(u(x)-u(x+y))K(y)\,dy,\label{eq:1.1a}
\end{equation}
where the kernel function $K(y)$ for $(-\Delta )^a$ equals
$c|y|^{-n-2a}$; they are generators of L\'evy processes. The other is
the definition as a {\em pseudodifferential
operator}
\begin{equation}
  Pu(x)=\mathcal
F^{-1}_{\xi \to x}\bigl(p(x,\xi )(\mathcal F u)(\xi )\bigr)=\operatorname{OP}(p)u(x),\label{eq:1.1b}
\end{equation}
where  $\mathcal F$ stands for
the Fourier transform; here $p(x,\xi )$ equals $|\xi
|^{2a}$ in the case of $(-\Delta )^a$; note that $|\xi |^{2a}= \mathcal
F(c|y|^{-n-2a})$.
The generalizations of \eqref{eq:1.1a} allow even functions $K(y)$
with less smoothness in $y$; here boundedness above
and below in
comparison with $|y|^{-n-2a}$ is usually assumed (a limited number of studies exist including $x$-dependence).
The generalizations based on \eqref{eq:1.1b} need specific smoothness
assumptions, particularly in $\xi $; however the theory
allows $x$-dependence in a systematic way. The two types have a
considerable overlap. The pseudodifferential methods made it possible
to determine the precise domain of the operator subject to a Dirichlet
condition  \cite{G15,AG23,G23} 

In the following we shall develop results
that primarily rely on pseudodifferential methods, but we shall also
take recourse to probabilistic results at a certain point.}

Optimal regularity results for solutions of linear parabolic equations
 {such as \eqref{eq:1.1}}
are essential for the construction of regular solutions of
corresponding nonlinear parabolic evolution equations with the aid of
the contraction mapping principle. Of particular importance are
results for $L_q$-based Sobolev type function space for general $q\in
(1,\infty)$ (not necessarily $q=2$) since in applications to nonlinear
equations one uses Sobolev type embeddings for $q$ sufficiently large.
This topic is intensively studied for parabolic differential
equations. But in the case of nonlocal operators in domains with
boundary there are only few results. This is of a particular challenge
since results on elliptic regularity in the standard spaces often
fail. 

Estimates of the solutions of \eqref{eq:1.1}  in $L_q$-based function spaces were shown by the second author \cite{G18a,G18b,G23} for $1<q<\infty$ in the case when $P$
is symmetric and  translation-invariant. The results were restricted to this case since the proofs relied on  a Markovian property
obtainable in that case. However, interior estimates (and global
estimates on $\R^{n+1}$) could be shown by
another method in $x$-dependent cases \cite{G18a}. We note that the works
\cite{G18a,G18b} assumed $\Omega $ to be $C^\infty $.

After the extension  in Abels-Grubb \cite{AG23} of the general
treatment of boundary problems for $P$  to cases with nonsmooth
domains $\Omega $, the heat equation results have been followed up in
\cite{G23}, the case $q\ne 2$ still limited to symmetric operators with the Markovian property.

In the present work we address the question of solvability of \eqref{eq:1.1} for
$x$-dependent operators $P=\operatorname{OP}(p(x,\xi ))$ in an
$L_q$-setting ($1<q<\infty $), when both $p$ and $\Omega $ are
nonsmooth. The symbols are assumed to be classical, strongly elliptic
and {\em even} (this is short for an alternating symmetry property of
the homogeneous terms \eqref{eq:3.1}), and the resolvent estimates are
obtained for a large class of nonsymmetric operators not necessarily
having the Markovian property. This includes the important example
$P=L^a$, where $L=-\nabla\cdot A(x)\nabla+b(x)$, $A(x)$ being a smooth $(n\times
n)$-matrix with positive lower bound, and $\operatorname{Re}b(x)\ge 0$; $A(x)$ is assumed real for
$x\in\partial\Omega $. Related operators are treated in a general framework on compact
boundaryless manifolds by Roidos and Shao \cite{RS22}.

We draw on several tools: The interior regularity is
obtained by the general strategy introduced in \cite{G18a} where a
symbolic calculus is set up for symbols with an extra parameter in the
style of \cite{GK93,G95,G96}, allowing the construction of a symbolic inverse (here
nonsmooth results may be included by a simple approximation). Another tool
is that the
resolvent estimates at the boundary can be obtained from the
$x$-independent case by the technique
presented in \cite[Section 6]{AG23}: Here the forward operator
$P-\lambda $ is compared to its principal part ``frozen'' at a
boundary point $x_0$, and an estimate can be obtained in a small
neighborhood of $x_0$ by a scaling that flattens the symbol of $P$ and
the boundary.

Still another tool to obtain sharp solvability properties in
$L_q$-spaces, is to aim for
$\mathcal R$-bounds on the resolvent of the Dirichlet realization of $P$
on $\Omega $. This has to the best of our knowledge not been attempted
before for 
these fractional-order problems. Here we use the
theory laid out in e.g.\ Denk-Hieber-Pr\"uss \cite{DHP03} and
Pr\"uss-Simonett \cite{PS16}.

The main linear results are:

\begin{thm}\label{thm:1.1} Let  $0<a<1$, $\tau >2a$,  $1<q<\infty $,
  and let $\Omega $ be bounded with  $C^{1+\tau
}$-boundary. Let
$P=\operatorname{OP}(p)$ with symbol $p(x,\xi )\in C^\tau
S^{2a}(\rn\times\rn)$, strongly elliptic and even, and
assume that the principal symbol $p_0(x,\xi )$ is real positive at
each boundary point   $x\in \partial\Omega$. Denote the
$L_q$-Dirichlet realization by $P_D$.

Then the resolvent $(P_D-\lambda )^{-1}$ exists for $\lambda $
in a set $V_{\delta ,K}$ with $0<\delta <
\frac\pi 2$, $K\ge 0$,
\begin{equation}
  \label{eq:1.2}
V_{\delta,K} =\{\lambda \in
\C\setminus \{0\}\mid   \operatorname{arg}\lambda \in [\tfrac \pi
2-\delta , \tfrac {3\pi }2+\delta ],
|\lambda |\ge K\},
\end{equation}
 and the operator family $\{\lambda
(P_D-\lambda )^{-1} \mid \lambda \in V_{\delta ,K} \}\subset \mathcal
L(L_q(\Omega ))$ is $\mathcal R$-bounded.  
\end{thm}
 {\begin{rem}
  The assumption that the principal symbol $p_0(x,\xi )$ is real positive at
each boundary point $x\in \partial\Omega$ is made for technical reasons. The proof is based on localization and perturbation arguements, where a maximal regularity result for constant coefficient operators with real and positive principle part is the starting point, cf.~Proposition~\ref{prop:5.2} below. 
\end{rem}}

The domain of $P_D$ is a so-called $a$-transmission space
$H_q^{a(2a)}(\comega)$ \cite{G15},\cite{AG23}, denoted $D_q(\comega)$
for short.

\begin{thm}\label{thm:1.2}
Assumptions as in Theorem~\ref{thm:1.1}. Let
$1<p<\infty $.

For any
$f\in L_p(I;L_q(\Omega))$, any $T>0$, the heat equation \eqref{eq:1.1} has  a unique solution $u\in
C^0(\overline{I};L_q(\Omega ))$  satisfying
\begin{equation}
  \label{eq:1.3}
u\in L_p(I;D_q(\overline \Omega ))\cap  H^1_p( I;L_q(\Omega
)).  
\end{equation}
\end{thm}

This is maximal $L_p$-regularity, shown here for the first time for
nonsymmetric 
fractional-order Dirichlet problems with $x$-dependent symbols.

There is also a solvability result with a nonzero local Dirichlet condition, when
$q<(1-a)^{-1}$:

\begin{thm}\label{thm:1.3}
In addition to the assumptions in Theorem~\ref{thm:1.1}, let $\tau >2a+1$,
$q<(1-a)^{-1}$  and $1<p<\infty $. The nonhomogeneous heat problem
\begin{equation}
  \label{eq:1.4}
  \begin{split}
    \partial_tu+Pu&=f\text{ on }\Omega \times I ,\\
     \gamma _0(u/d_0^{a-1})&=\psi \text{ on }\partial\Omega \times I ,\\
u&=0\text{ on }(\R^n\setminus\Omega )\times I, \\
u|_{t=0}&=0,
  \end{split}
\end{equation}
has for  $f\in
L_p(I;L_q(\Omega ))$, $\psi \in L_p(I;B_{q,q}^{a+1-1/q}(\partial\Omega ))\cap  
H_p^1(  I; B_{q,q}^\varepsilon (\partial\Omega ))$ with $\psi
 (x,0)=0$ ($\varepsilon >0$), and any $T>0$ a unique
 solution $u$ satisfying
 \begin{equation*}
u\in L_p(I;H_q^{(a-1)(2a)}(\comega ))\cap   H_p^1(I; L_q(\Omega )).
\end{equation*}
\end{thm}
 {\begin{rem}
  We note that the assumption $\psi \in L_p(I;B_{q,q}^{a+1-1/q}(\partial\Omega ))\cap  
H_p^1(  I; B_{q,q}^\varepsilon (\partial\Omega ))$ is not optimal. The statement of Theorem~\ref{thm:1.3} holds true for any $\psi$ in the trace space of $L_p(I;H_q^{(a-1)(2a)}(\overline{\Omega} ))\cap   H_p^1(I; L_q(\Omega ))$ with respect to $\gamma _0(\cdot /d_0^{a-1})$. But we do not have a characterization of this space for the time being. 
\end{rem}}

From Theorem \ref{thm:1.2} we moreover deduce nonlinear
results. Consider the parabolic problem
\begin{equation}\label{eq:1.5}
  \begin{aligned}
  \partial_tu+ a_0(x,u)Pu &= f(x,u)&\quad& \text{in }\Omega\times (0,T),\\
  u&=0 &\quad&\text{on } (\R^n\setminus \Omega)\times (0,T),\\
  u|_{t=0} &= u_1 &\quad& \text{in }\Omega,
  \end{aligned}
\end{equation}
for some $T>0$. 

\begin{thm}\label{thm:Nonlinear'}
   Let $\Omega $ be a bounded domain with  $C^{1+\tau
   }$-boundary for some $\tau >2a$, and let $1<p,q<\infty $ be such that
       $ (a+\tfrac1q)(1-\tfrac{1}p) -\tfrac{n}q >0$.
If $n=1$, assume moreover $\frac1q<a$. Let
$P$ be as in Theorem \ref{thm:1.1}.
 Moreover, for an open set 
$U\subset \R$  with $0\in U$,  let $a_0\in
C^{\max(1,\tau)}(\R^n\times U,\R)$ with $a_0(x,s)>0$ for all $s\in U$ and
$x\in\R^n$, let $f\colon \Rn\times U\to \R$ be continuous in  $(x,u)$ and locally Lipschitz in $u$, and let
 {$u_0 \in X_{\gamma,1}
  \cap
C^\tau(\overline{\Omega}) $ with $\overline{u_0(\Omega)}\subset
U$; here $X_{\gamma,1}:=
(L_q(\Omega),D_q(\overline\Omega))_{1-\frac1p,p}$.

Then there are $\eps_0,T>0$ such that for every
$u_1\in X_{\gamma,1}$} with
$\|u_0-u_1\|_{X_{\gamma,1}}\leq \eps_0$, the system
\eqref{eq:1.5} possesses a unique solution
\begin{equation*}
  u\in L_p((0,T);D_q(\overline{\Omega}))\cap H^1_p((0,T);L_q(\Omega)).
\end{equation*}
\end{thm}

This leads in particular to  solvability results for nonlinear
diffusion equations, including problems of the type of the
porous medium equation, see Corollary \ref{cor:Diffusion}ff.

Earlier works on \eqref{eq:1.1}
have mostly been concerned with $P=(-\Delta )^a$ and
$x$-independent singular integral operator generalizations. To mention a few: There are results on Schauder
estimates and H\"older properties, by e.g.\
Felsinger and Kassmann \cite{FK13}, Chang-Lara and Davila \cite{CD14},
Jin and Xiong \cite{JX15}; and quite precise results on regularity in anisotropic H\"older spaces
by Fernandez-Real and Ros-Oton \cite{FR17}, and Ros-Oton and Vivas \cite{RV18}. For $P=(-\Delta
)^a$, Leonori, Peral, Primo
and Soria \cite{LPPS15} showed $L_q(I;L_r(\Omega ))$ estimates; Biccari, Warma
and Zuazua \cite{BWZ18} showed 
$L_q(I;B^{2a}_{q,r,loc}(\Omega ))$-estimates for certain $r$, and
Choi, Kim and Ryu have in \cite{CKR23} shown
weighted $L_q$-esti\-ma\-tes. Results on $\rn$ with $x$-dependence
have been obtained by by Dong, Jung and Kim \cite{DJK23}.
 {Singular integral formulations with $x$-dependence are
  presented in a systematic way by Fern\'{a}ndez-Real and Ros-Oton in \cite{FR24}.}
As recalled
further above, we have earlier shown maximal $L_q$-regularity
results on $\rn$, and for domains $\Omega $ in the
translation-invariant symmetric case, \cite{G18a},
\cite{G23}. Roidos and Shao \cite{RS22} show maximal $L_p$-regularity for
operators like $(-\nabla\cdot \mathfrak a(x)\nabla)^a$ on compact boundaryless
manifolds. The latter includes nonlinear applications such as the fractional
porous medium equation; this is also treated in  V\'azques, de~Pablo, Quir\'{o}s and Rodr\'{\i}guez 
\cite{VPQR17} on $\rn$ in H\"older spaces.

\medskip

\noindent{\it Plan of the paper:} Section 2 recalls definitions of function
spaces and pseudodifferential operators. Section 3 presents our
hypotheses on $P$ and sets up the Dirichlet
realization. Section 4 collects the needed features
of $\mathcal R$-boundedness and their connection with maximal $L_p$-regularity. In Section 5, we deduce the main results on
resolvent estimates for the Dirichlet realization. In Section 6, this is applied to obtain
the linear results for time-dependent problems. Finally, in Section 7
we show some consequences for nonlinear evolution problems, including
  fractional nonlinear diffusion equations and  fractional porous medium equations.

\section{Prerequisites}
\subsection{Function spaces}

The space $C^k(\rn)\equiv C^k_b(\rn)$ consists  of $k$-times differentiable
functions with bounded norms $\|u\|_{C^k}=\sup_{|\alpha |\le k,x\in\rn}|D^\alpha
u(x)|$ ($k\in{\mathbb N}_0$), and the H\"older
spaces $C^\tau   (\rn)$, $\tau  =k+\sigma $ with $k\in{\mathbb N}_0$,
$0<\sigma <1$, also denoted $C^{k,\sigma } (\rn)$, consists of
functions $u\in C^k(\rn)$ with bounded norms
$\|u\|_{C^\tau  }=\|u\|_{C^k}+\sup_{|\alpha |= k,x\ne y}|D^\alpha
u(x)-D^\alpha u(y)|/|x-y|^\sigma $. The latter definition extends to
Lipschitz spaces $C^{k,1} (\rn)$. There are similar spaces over subsets
of $\rn$. 
There are also the H\"older-Zygmund spaces $C^s _*(\rn)\equiv B^s
_{\infty ,\infty }(\rn)$ defined for $s\in\R$ with good interpolation
properties, coinciding with $C^s(\rn)$ when $s\in\rp\setminus \N$.

The halfspaces $\rnpm$ are defined by
 $\rnpm=\{x\in
{\mathbb R}^n\mid x_n\gtrless 0\}$, with points denoted  $x=(x',x_n)$,
$x'=(x_1,\dots, x_{n-1})$; $\R^1_\pm$ is denoted $\R_\pm$. For a given real function $\zeta \in C^{1+\tau}
(\R^{n-1})$ (some $\tau >0$), we define the curved halfspace
$\rn_\zeta  $ by
\begin{equation}
  \label{eq:2.1}
\rn_\zeta = \{x\in\R^n\mid x_n>\zeta (x')\};
\end{equation}
it is
 a $C^{1+\tau }$-domain.

By a bounded $C^{1+\tau }$-domain $\Omega $  we mean the following:
$\Omega \subset\rn$ is open,
bounded and nonempty,
and every boundary point
$x_0$ has an open neighborhood $U$ such that, after a translation of
$x_0$ to $0$ and a suitable rotation, $U\cap \Omega $ equals $U\cap \rn_\zeta 
$ for a function $\zeta  \in C^{1+\tau }(\R^{n-1})$ with $\zeta
(0)=0$.

Restriction from $\R^n$ to $\rnpm$ (or from
${\mathbb R}^n$ to $\Omega $ resp.\ $\complement\comega= \rn \setminus \comega$) is denoted $r^\pm$,
 extension by zero from $\rnpm$ to $\R^n$ (or from $\Omega $ resp.\
 $\complement\comega$ to ${\mathbb R}^n$) is denoted $e^\pm$. (The
 notation is also used for $\Omega =\rn_\zeta  )$.) Restriction
 from $\crnp$ or $\comega$ to $\partial\rnp$ resp.\ $\partial\Omega $
 is denoted $\gamma _0$.

When $\Omega $ is a $C^{1+\tau }$-domain, we denote by $d(x)$ 
  a function  that is $C^{1+\tau } $ on $\comega$,
positive on $\Omega $ and vanishes only to the first order on
$\partial\Omega $ (i.e., $d(x)=0$ and $\nabla d(x)\neq 0$ for $x\in
\partial\Omega$). It is bounded above and below by the distance
$d_0(x)=\dist(x,\partial\Omega )$; see further details in \cite{AG23}.

Throughout the paper, $q$ satisfies $1<q<\infty $. The Bessel-potential spaces
$H^s_q({\mathbb R}^n)$ are defined
for $s\in{\mathbb R}$ by 
\begin{equation}
  \label{eq:2.2}
H_q^s(\R^n)=\{u\in \mathcal S'({\mathbb R}^n)\mid \F^{-1}(\ang{\xi }^s\hat u)\in
L_q(\R^n)\},
\end{equation}
where $\mathcal F$ is the Fourier transform  $\hat
u(\xi )=\mathcal F
u(\xi )= \int_{{\mathbb R}^n}e^{-ix\cdot \xi }u(x)\, dx$, and the
function $\ang\xi $ equals $(|\xi |^2+1)^{\frac12}$.  For  $q=2$, this is the scale of $L_2$-Sobolev spaces,
where the index $2$ is usually omitted. $\mathcal S'(\rn)$ is the Schwartz
space of temperate distributions, the dual space of $\mathcal S(\rn)$;
the space of rapidly decreasing $C^\infty $-functions. (The spaces can
be defined for other values of $q$, but some properties we need are
linked to $q\in (1,\infty )$.)

For $s\in {\mathbb
N}_0=\{0,1,2,\dots\}$, the spaces $H_q^s(\rn)$ are also denoted $W_q^s({\mathbb R}^n)$ or $W^{s,q}({\mathbb R}^n)$
in the literature. We moreover need to refer to the Besov
spaces  $B^s_{q,q}(\rn)$, also denoted $B^s_q(\rn)$,
that coincide with the
$W^s_q$-spaces when $s\in \rp\setminus \N$.
They necessarily enter in connection with
 boundary value problems in an $H^s_q$-context,
 because they are the correct range spaces for
trace maps $\gamma _ju=(\partial_n^ju)|_{x_n=0}$:
\begin{equation}
  \label{eq:2.3}
\gamma _j\colon \ol H^s_q(\rnp), \ol B^s_{q,q}(\rnp) \to
B_{q,q}^{s-j-\frac1q}({\mathbb R}^{n-1}), \text{ for }s-j-\tfrac1q >0,
\end{equation}
 (cf.\ \eqref{eq:2.4}), surjectively and with a continuous right inverse; see e.g.\ the overview in
the introduction to \cite{G90}. For $q=2$, the two scales $H_q^s$ and
$B_{q,q}^s$ are
identical, but for $q\ne 2$ they are related by strict inclusions: $
H^s_q\subset B^s_{q,q}\text{ when }q>2$, $H^s_q\supset B^s_{q,q}\text{ when
}q<2$.

In relation to $\Omega $, \eqref{eq:2.2} gives rise to 
 two scales of spaces for $s\in{\mathbb R}$:
 \begin{equation}
   \label{eq:2.4}
\begin{split}
\ol H_q^{s}(\Omega)&=\{u\in \mathcal D'(\Omega )\mid u=r^+U \text{ for some }U\in
H_q^{s}(\R^n)\}, \text{ the {\it restricted} space},\\
\dot H_q^{s}(\comega)&=\{u\in H_q^{s}({\mathbb R}^n)\mid \supp u\subset
\comega \},\text{ the {\it supported} space;}
\end{split}
 \end{equation}
here $\operatorname{supp}u$ denotes the support of $u$ (the complement
 of the largest open set where $u=0$). 
 $\ol H_q^s(\Omega )$ is in other texts often denoted  $H_q^s(\Omega )$  or
$H_q^s(\comega )$, and $\dot H_q^{s}(\comega)$ may be indicated with a
ring, zero or twiddle;
the current notation stems from H\"ormander \cite[App. B.2]{H85}.
For $1<q<\infty $, there is an identification of $\ol H_q^s(\Omega )$ with the dual space
of $\dot H_{q'}^{-s}(\comega)$, $\frac1{q'}=1-\frac1q\,$, in terms of a
duality extending the sesquilinear scalar product
$\ang{f,g}=\int_\Omega f\,\ol g\, dx$.

In discussions of heat operator problems it will sometimes be
convenient to refer to anisotropic Bessel-potential spaces over
$\R^{n+1}=\{(x,t)\mid x\in\rn, t\in\R\}$. With $d\in\rp$, we define
\begin{equation}
  \label{eq:2.5}
\{\xi ,\tau \}\equiv (\ang\xi ^{2d}+\tau ^2)^{1/(2d)},
\end{equation}
leading to the ``order-reducing'' operators (defined for all $s\in{\mathbb R}$)
$$
\Theta ^su=\Op(\{\xi ,\tau \}^s)u\equiv \F^{-1}_{(\xi ,\tau )\to
(x,t)}(\{\xi ,\tau \}^s\F_{(x,t)\to (\xi ,\tau )}u),
$$
 Then we define:
 \begin{equation}
   \label{eq:2.6}
H^{(s,s/d)}_q(\rn\!\times\!\mathbb R)=\Theta ^{-s}L_q
(\mathbb R^{n+1});
 \end{equation}
for $1< q<\infty $, $s\in\R$.  Note that
the case $s=0$ gives $L_q(\R^{n+1})$, and the case $s=d$ gives
\begin{equation}
  \label{eq:2.7}
H^{(d,1)}_q(\rn\!\times\!\mathbb R)= L_q(\R;H_q^d(\rn))\cap H^1_q(\R;
L_q(\rn)).
\end{equation}
More on these spaces in \cite{G18a}.

\subsection{Pseudodifferential operators  and transmission spaces}

Recall that
the pseudodifferential operator (briefly expressed: $\psi$do) $P$ on $\rn$ with symbol $p\colon \rn\times \rn\to \C$ is defined as 
\begin{equation}\label{eq:2.8}
(Pu)(x)=\mathcal
F^{-1}_{\xi \to x}\bigl(p(x,\xi )(\mathcal F u)(\xi )\bigr)=\operatorname{OP}(p)u(x),
\end{equation}
where  $(\mathcal F u)(\xi )= \hat
u(\xi )=
\int_{{\R}^n}e^{-ix\cdot \xi }u(x)\, dx
$ denotes the Fourier transform of $u$ for suitable $u\colon \rn\to \C$, and $\mathcal F^{-1}$ is the inverse Fourier transform.
Under suitable conditions on the symbol $p$, $P$ is well-defined for
$u\in \mathcal S(\rn)$, and the definition extends to much more general
spaces.
(Further details and references are given in \cite{AG23}, \cite{G23}.)

 For $\tau >0$, $m\in \R$, the
space    $C^\tau S^m_{1,0}(\R^{n}\times \rn)$ of $C^\tau $-symbols of
order $m$
 consists of functions
$p \colon  \R^{n}\times \R^n\to \C$ that are continuous with respect
to 
$(x,\xi)\in  \R^{n} \times \R^n$ and $C^\infty $ with respect to $\xi\in
\R^n$, such that for every $\alpha\in\N_0^n$  we have:
$\partial_\xi^\alpha p(x,\xi)$ is in $C^\tau (\R^{n})$ with respect to
$x$ and satisfies for all $\xi\in\R^n$, $\alpha \in {\mathbb N}_0^n$, 
\begin{equation}
  \label{eq:2.9}
  \|\partial_\xi^\alpha p(\cdot,\xi)\|_{C^\tau  (\R^{n})}\leq C_\alpha
  \ang{\xi}^{m-|\alpha|} ,
\end{equation}
for some $C_\alpha>0$. The symbol space is a Fr\'echet space with the semi-norms
\begin{equation}
  \label{eq:2.10}
  |p|_{k,C^\tau S^m_{1,0}(\R^{n}\times \R^n)}:= \max_{|\alpha|\leq k} \sup_{\xi\in\R^n} \ang{\xi}^{-m+|\alpha|}\|\partial_\xi^\alpha p(\cdot,\xi)\|_{C^\tau  (\R^{{n}})}\quad \text{for }k\in\N_0.
\end{equation}
For such symbols there holds:
\begin{equation}
  \label{eq:2.11}
    \Op(p)\colon H^{s+m}_q (\R^n)\to H^s_q(\R^n)\quad \text{for all
    }|s|<\tau,
    \end{equation}
where the operator norm for each $s$ 
is estimated by a semi-norm for some $k\in\mathbb{N}_0$ (depending on
$s$).

The space of $C^\infty $-symbols  $S^m_{1,0}({\mathbb R}^n\times{\mathbb
R}^n)$ of order $m\in{\mathbb R}$ equals  $\bigcap _{\tau >0}C^\tau
S^m_{1,0}({\mathbb R}^n\times{\mathbb R}^n)$.

The subspaces of \emph{classical} symbols  $C^\tau S^m(\R^{n}\times \rn)$ resp.\  $
S^m(\R^{n}\times \rn)$ consist of those functions in  $C^\tau S^m_{1,0}(\R^{n}\times \rn)$ resp.\  $
S^m_{1,0}(\R^{n}\times \rn)$ that moreover have
expansions into terms  $p_j$ homogeneous in $\xi $ of degree $m-j$ for
$|\xi |\ge 1$, all $j\in\N_0$, such that for all $\alpha
\in{\mathbb N}_0^n$, $ J\in{\mathbb N}_0$ there is some $C_{\alpha,J}>0$ satisfying
\begin{equation}
  \label{eq:2.12}
\|\partial_\xi ^\alpha \bigl(p(\cdot,\xi )-
{\sum}_{j<J}p_j(\cdot,\xi )\bigr)\|_{C^\tau  (\R^{n})}\leq C_{\alpha
,J}\ang{\xi}^{m-J-|\alpha|}\quad \text{for all }\xi\in\R^n.
\end{equation}

The operator $P=\Op(p)$ and the symbol $p$ are said to be \emph{elliptic}, when, for a sufficiently large $R>0$
there is a $c>0$ such that
$$
|p(x,\xi )|\ge c|\xi |^m \quad \text{for all }|\xi |\ge R, x\in\rn;
$$
this holds in the classical case if and only if (with some $c'>0$)
$$
|p_0(x,\xi )|\ge c'|\xi |^m \quad\text{for all }|\xi |\ge 1, x\in \rn.
$$

A special role in the theory is played by the \emph{order-reducing
operators}. There is a simple definition of operators $\Xi _\pm^t $ on
${\mathbb R}^n$, $t\in{\mathbb R}$,
\begin{equation}
  \label{eq:2.13}
\Xi _\pm^t =\operatorname{OP}(\chi _\pm^t),\quad \chi _\pm^t=(\ang{\xi '}\pm i\xi _n)^t ;
\end{equation}
 they preserve support
in $\crnpm$, respectively. The functions
$(\ang{\xi '}\pm i\xi _n)^t $ do not satisfy all the estimates for $S^{t }_{1,0}({\mathbb
R}^n\times{\mathbb R}^n)$, but  definition \eqref{eq:2.8} applies anyway. There is a more refined choice $\Lambda _\pm^t $
\cite{G90, G15}, with
symbols $\lambda _\pm^t (\xi )$ that do
satisfy all the estimates for $S^{ t }_{1,0}({\mathbb
R}^n\times{\mathbb R}^n)$; here $\overline{\lambda _+^t }=\lambda _-^{t }$.
The symbols have holomorphic extensions in $\xi _n$ to the complex
halfspaces ${\mathbb C}_{\mp}=\{z\in{\mathbb C}\mid
\operatorname{Im}z\lessgtr 0\}$; it is for this reason that the operators preserve
support in $\crnpm$, respectively. Operators with that property are
called ``plus'' resp.\ ``minus'' operators. There is also a pseudodifferential definition $\Lambda
_\pm^{(t )}$ adapted to the situation of a smooth domain $\Omega
$, cf.\ \cite{G15}. For nonsmooth domains, one applies the operators
$\Xi ^t_\pm$ in localizations where a piece of $\Omega $
is carried over to a piece of  $\rnp$.

It is elementary to see by the definition of the spaces $H_q^s(\R^n)$
in terms of Fourier transformation, that the operators define homeomorphisms 
for all $s$: $
\Xi^t _\pm\colon H_q^s(\R^n) \simto H_q^{s- t
}(\R^n)$, $  
\Lambda ^t _\pm\colon H_q^s (\R^n) \simto H_q^{s- t
} (\R^n)$.
The special
interest is that the ``plus''/''minus'' operators also 
 define
homeomorphisms related to $\crnp$ and $\comega$, for all $s\in{\mathbb R}$: 
$
\Xi ^{t }_+\colon \dot H_q^s(\crnp )\simto
\dot H_q^{s- t }(\crnp)$, $
r^+\Xi ^{t }_{-}e^+\colon \ol H_q^s(\rnp )\simto
\ol H_q^{s- t } (\rnp )$, with similar statements  for $\Lambda ^t_\pm$,
and for $
\Lambda^{(t )}_\pm$ relative to $\Omega $.
Moreover, the operators $\Xi ^t _{+}$ and $r^+\Xi ^{t }_{-}e^+$ identify with each other's adjoints
over $\crnp$, because of the support preserving properties.
There is a
similar statement for $\Lambda ^t_+$ and  $r^+\Lambda ^t_-e^+$, and for $\Lambda ^{(t )}_+$ and $r^+\Lambda ^{(
t )}_{-}e^+$ relative to the set $\Omega $.

The special {\it $\mu $-transmission spaces} were 
introduced by
H\"ormander \cite{H66} for $q=2$, and developed  in detail for
$1<q<\infty $ by Grubb \cite{G15}:
\begin{equation}
  \label{eq:2.14}
  \begin{split}
H_q^{\mu  (s)}(\crnp)&=\Xi _+^{-\mu  }e^+\ol H_q^{s- \mu 
}(\rnp)=\Lambda  _+^{-\mu  }e^+\ol H_q^{s- \mu 
}(\rnp)
,\quad \text{if }  s> \mu  -\tfrac1{q'},\\
H_q^{\mu  (s)}(\comega)&=\Lambda  _+^{(-\mu  )}e^+\ol H_q^{s- \mu 
}(\Omega ),\quad\text{if }  s> \mu  -\tfrac1{q'};
  \end{split}
\end{equation}
here $\mu >-1$.
With $\mu =a$, they are the appropriate solution spaces for homogeneous Dirichlet
problems for the operators of order $2a$ that we shall study. For
problems with a nonhomogeneous local Dirichlet condition they enter
with $\mu =a-1$. There holds $H_q^{\mu  (s)}(\comega)\subset H_q^{\mu
  (s')}(\comega)$ for $s>s'$.
In the
first line of \eqref{eq:2.14}, we have
\begin{equation}
  \label{eq:2.14a}
H_q^{\mu  (s)}(\crnp)=\Xi _+^{-\mu  }e^+\ol H_q^{s- \mu 
}(\rnp)=\Xi _+^{-\mu  }\dot H_q^{s- \mu 
}(\crnp)=\dot H_q^{s}(\crnp)
,\quad \text{if } \mu +\tfrac1q> s >\mu  -\tfrac1{q'}.
\end{equation}
On the other hand, when
$s>\mu +\tfrac1q$,
  $\Xi
_+^{-\mu }$ is applied to functions having a jump at $x_n=0$; this results in a singularity $x_n^\mu $ at $x_n=0$. 

The second line in \eqref{eq:2.14} is valid in the case of a $C^\infty
$-domain $\Omega $. In the case where $\Omega $ is $C^{1+\tau }$,
$\tau >0$, we have instead a definition using local coordinates, based on the
definition for the case of a curved halfspace $\R^n_\zeta $
\eqref{eq:2.1}. Here we use the diffeomorphism $F_\zeta $ mapping $\rn_\zeta $
to $\rnp$ and its inverse $F_\zeta ^{-1}$,
$$
F_\zeta (x)=(x',x_n-\zeta (x')),\quad F^{-1}_\zeta (x)=(x',x_n+\zeta (x')),\quad 
$$
defining, for $\mu -\frac1{q'}<t<1+\tau $,
$$
u\in H_q^{\mu (t)}(\ol {\R}_\zeta ) \iff u\circ F_\zeta ^{-1}\in H_q^{\mu (t)}(\crnp),
$$
with the inherited norm ($u\circ F_\zeta ^{-1}$ is also denoted  $F_\zeta ^{-1,*}u$).
For a bounded $C^{1+\tau }$-domain $\Omega $, every point
 $x_0\in \partial\Omega $ has a bounded open neighborhood $U\subset
 \rn$ and a $\zeta \in C^{1+\tau }(\R^{n-1})$, such that after a
 suitable rotation, $\Omega \cap U=\rn_\zeta \cap U$.
 $H_q^{\mu (t)}(\comega)$ is now defined (cf.\ \cite[Def.\ 4.3]{AG23})
  as the set
 of functions $u\in H^t_{loc}(\Omega )$ such that for each $x_0$, with
 a $\varphi \in C_0^\infty (U)$ with $\varphi \equiv 1$ in a
 neighborhood of $x_0$, $(\varphi u)\circ F_\zeta ^{-1}\in
 H_q^{\mu (t)}(\crnp)$ (in the rotated situation).

A norm on $H_q^{\mu (t)}(\comega)$ can be defined as follows: There is a
cover of $\comega$ by bounded open sets $\{U_0,U_1,\dots, U_J\}$, and a partition
of unity $\{\varrho _j\}_{0\le j\le J}$ (with $ \varrho _j\in
C_0^\infty ( U_j, [0,1])$, satisfying $\sum_{0\le j\le J}\varrho _j=1$ on
$\comega$), where the $U_j$
for $j\ge 1$ are neighborhoods of  points $x_j\in \partial\Omega $
with $\Omega \cap U=\rn_{\zeta _i}\cap U$ (after a rotation), $\zeta
_i\in C^{1+\tau }(\rn)$, as described above. Moreover, $\partial\Omega
\subset \bigcup_{1\le j\le J}U_j $ and 
$ \ol U_0\subset \Omega $. Then 
\begin{equation}
  \label{eq:2.15}
\|u\|_{H_q^{\mu (t)}(\comega)}= \bigl(\sum_{1\le j\le J}\|(\varrho _ju)\circ F_{\zeta _j}^{-1}\|^q_{H_q^{\mu (t)}(\crnp)}+\|\varrho _0u\|^q_{H^t_q(\rn)}\bigr)^{\frac1q}
\end{equation}
is a norm on ${H_q^{\mu (t)}(\comega)}$. (This way to define norms over
curved spaces is
recalled e.g.\ in \cite[Sect.\ 8.2]{G09}.)

Further properties of the $\mu $-transmission spaces are described in
detail in \cite{G15}, \cite{G19}, \cite{AG23} and \cite{G23}.

\section{The Dirichlet realization}
Our main hypothesis on $P$ is: 
\begin{hypo}\label{hypo:3.1} 
 Let  $0<a<1$, $\tau >2a$, and $P=\Op(p)$,
where $p\in C^\tau
S^{2a}(\rn\times\rn)$ (is a classical $C^\tau $-symbol of order $2a$).
Moreover, $P$ is strongly elliptic, i.e., $\operatorname{Re}p_0(x,\xi
)\ge c|\xi |^{2a}$ with $c>0$ for $|\xi |\ge 1$, and has the evenness
property:
\begin{equation}
  \label{eq:3.1}
  p_j(x,-\xi )=(-1)^jp_j(x,\xi )  \text{ for all }
  j\in\N_0,\; |\xi |\ge 1,\; x\in\rn.
  \end{equation}
\end{hypo}
\begin{rem}\label{rem:3.2}
One of the convenient properties of the pseudodifferential calculus is that for
elliptic problems, the interior regularity of solutions is dealt with,
once and for all: When $P$ is classical elliptic (i.e., $p_0(x,\xi )\ne
0$ for $|\xi |\ge 1$) of order $m$, then for any open set $\Omega
\subset \rn$, $Pu|_{\Omega }\in
H^s_{q,loc}(\Omega )$ implies  $u|_{\Omega }\in
H^{s+m}_{q,loc}(\Omega )$.
In the case $\tau =\infty $, this holds for $s\in\R$ and was shown already
by Seeley in \cite{S65a} (see also \cite{S65b}) in $H^s_q$-spaces, and it extends to
all scales of function spaces, where pseudodifferential operators are continuous, as indicated in
\cite{G14}. For finite $\tau $ and, say,  $u\in
H^{m-\tau+\varepsilon}_q(\rn)$ for some $\varepsilon>0$, it
follows for $-\tau < s\leq \tau $ e.g.\ from Theorem~9 in Marschall \cite{M92}
after $P$ is reduced to order zero and a standard localization
procedure is applied.
\end{rem}

Now some words on the special case where the symbol is real, $x$-independent and has no lower-order terms,
$p(x,\xi )=p_0(\xi )$. Denote by $p^h\colon \rn\to \C, \xi \mapsto p^h(\xi)$ the homogeneous function
on $\rn$ coinciding with $p_0$ for $|\xi |\ge 1$. The operator
$P^h=\Op(p^h)$ is then a complexified version of the real singular integral operator
$\mathcal L$
studied in many works on generalizations of the fractional Laplacian
(cf.\ e.g.\ Ros-Oton et al.\ \cite{R16}, \cite{RSV17}):
\begin{align}
  \nonumber
  \mathcal Lu(x)&=PV\int_{\rn}(u(x)-u(x+y))K(y)\,dy\\\label{eq:3.2}
                &=\int_{\rn}(u(x)-\tfrac12(
                  u(x+y)+u(x-y)))K(y)\,dy.
                  \end{align}
Here $K\colon \rn\setminus \{0\}\to \C$ is homogeneous of degree $-n-2a$ and smooth in $\rn\setminus \{0\}$ when $p^h$ is so, and even:
$K(-y)=K(y)$ for all $y\neq 0$. In the rotation-invariant case, $\mathcal L=(-\Delta )^{a}$
when $K(y)=c_{n,a}|y|^{-n-2a}$ for a suitable $c_{n,a}>0$. And more generally, this singular integral
definition coincides with our pseudodifferential definition of $P^h$,
when $K=(\mathcal F^{-1}p^h)|_{\rn \setminus \{0\}}$. Note here that $P^h$ differs from $P_0=\Op(p_0(\xi ))$ by
the operator $\mathcal R=\Op(r)$, where $r=p^h-p_0$ is bounded
and supported for $|\xi |\le 1$, hence $\mathcal R$ maps e.g.\ $H_q^s(\rn)\to
H_q^t(\rn)$ for all $s,t\in\R$; it is smoothing. So mapping properties
and regularity results for $\mathcal L$ follow from those for $P_0$
(or $P^h$).

Let $\Omega \subset \rn$ be open and bounded
with a $C^{1+\tau }$-boundary. The homogeneous Dirichlet problem for
$P$ is, for a given function $f$ on $\Omega $ to find  $u$ such that
\begin{equation}
  \label{eq:3.3}
Pu=f\text{ on }\Omega ,\quad u=0\text{ on }\rn\setminus \Omega.    
\end{equation} 
(More precisely, one can write $r^+Pu$ instead of ``$Pu$ on $\Omega
$''.)

From the sesquilinear form $s(u,v)$ obtained by closure on $ \dot
H^a(\overline\Omega )$ of
\begin{equation}
  \label{eq:3.4}
s(u,v)=\int_{\Omega }Pu\,\bar v\, dx,\quad u,v\in C_0^\infty (\Omega
), 
\end{equation}
one defines the Dirichlet realization $P_{D,2}$ in $L_2(\Omega )$ by
the Lax-Milgram lemma. For a general $1<q<\infty $, one likewise 
defines a Dirichlet realization $P_{D,q}$ of $P$ in
$L_q(\Omega )$, namely as the operator acting like $r^+P$ with domain 
 $
D(P_{D,q})=\{u\in \dot H_q^a(\overline\Omega )\mid r^+Pu\in L_q(\Omega
)\}$.
It is shown in \cite{G15} for $\tau =\infty $, \cite{AG23}
for general $\tau >2a$, that these operators have nice solvability
properties,
and their domains are found to equal $a$-transmission spaces
\begin{equation}
  \label{eq:3.5}
D(P_{D,q})=\{u\in \dot H_q^a(\overline\Omega )\mid r^+Pu\in L_q(\Omega
)\}=H_q^{a(2a)}(\comega).  
\end{equation}
By the observations around \eqref{eq:2.14a},
\begin{equation}
  \label{eq:3.5a}
H_q^{a(2a)}(\comega)=\dot H_q^{2a}(\comega)\text{ when }a<\tfrac1q;\quad
H_q^{a(2a)}(\comega)\subset \dot H_q^{a+\frac1q-\varepsilon
}(\comega)\text{ when }a
\ge \tfrac1q,  
\end{equation}
any $\varepsilon >0$. Moreover,  $H_q^{a(2a)}(\comega)$ is when $a>\tfrac1q$ contained in
$\dot H_q^{2a}(\comega)+d^ae^+\ol H^a_q(\Omega )$; recall that $d(x)\sim
\operatorname{dist}(x,\partial\Omega )$.
There is also an exact description
when $1+\frac1q>a>\frac1q$, namely: When $\tau =\infty $, $H_q^{a(2a)}(\comega)=\dot
H_q^{2a}(\comega)+d^aK_0B_{q,q}^{a-\frac1q}(\partial\Omega )$ by \cite{G19},
where $K_0$ is a Poisson operator and $B_{q,q}^{a-\frac1q}(\partial\Omega )$
is a Besov space; and this holds in local coordinates when $\tau $ is finite.
For brevity, we shall use the notation
\begin{equation}
  \label{eq:3.6}
D_q(\comega)= H_q^{a(2a)}(\comega). 
\end{equation}

In the following, we mostly consider a fixed $q$, and denote $P_{D,q}=P_D$.
$P_D$ has a resolvent that is compact in $L_q(\Omega )$, and as accounted
for in  \cite{G23}, the spectrum is contained in a convex sectorial region
opening to the right. Hence the resolvent set $\varrho (P_D)$ contains an obtuse sectorial
region $V_{\delta ,K}$ (the complement of a ``keyhole region''). Here we define  $V_{\delta ,K}$ 
for $0<\delta <\delta _0$ where $0<\delta _0\le
\frac\pi 2$, and $K\ge 0$, by
\begin{equation}
  \label{eq:3.7}
V_{\delta,K} =\{\lambda \in
\C\setminus \{0\}\mid   \operatorname{arg}\lambda \in [\tfrac \pi
2-\delta , \tfrac {3\pi }2+\delta ],
|\lambda |\ge K\}.
\end{equation}
For the actual $P$, $\tfrac\pi 2-\delta _0=\sup \{|\arg p_0(x,\xi )|\mid
x\in\rn, |\xi |\ge 1\}$.

In the $x$-independent real homogeneous  case considered around \eqref{eq:3.2} where
$P^h=\Op(p^h)$ 
identifies with $\mathcal L$, the quadratic form $s^h(u,u)$ (as in \eqref{eq:3.4})
identifies with the form 
\begin{equation}
  \label{eq:3.8}
Q(u)=\tfrac12\int_{\R^{2n}}|u(x)-u(y)|^2K(x-y)\,dxdy\quad \text{for } u\in \dot H^a(\comega),  
\end{equation}
acting on real $u$, cf.\ e.g.\ \cite{R16}.

\section{Auxiliary results on resolvent estimates}\label{sec:4}

Consider a  closed, densely defined linear operator $A\colon
\mathcal{D}(A)\subset X\to X$ on a UMD-space $X$. (Cf.\ e.g.\ Burkholder
\cite{B86}
for the definition and characterizations of UMD-spaces; the
$H_q^s$-spaces are of this kind.) Numerous studies through the times show that estimates of the
resolvent $(A-\lambda )^{-1}$  lead to solvability properties, in
various function spaces, of the heat equation
\begin{equation}
  \label{eq:4.1}
\partial_tu(t)+Au(t)=f(t)\text{ for }t\in I,\quad u(0)=0, 
\end{equation}
where $I=(0,T)$ for $T\in (0,\infty)$ or $I= (0,\infty)$.
A basic problem  is to show \emph{maximal
  $L_p$-regularity}, namely that \eqref{eq:4.1} for any $f\in L_p(I;X)$ has a unique solution $u\colon \overline{I} \to X$ satisfying
\begin{equation}
  \label{eq:4.2}
\partial_tu\text{ and } Au \in L_p(I;X). 
\end{equation}
We note that, if $I=(0,T)$ for some $T<\infty$, then this is equivalent to $u\in H^1_p(I;X)\cap L_p(I;\mathcal{D}(A))$.
This is usually relatively easy to obtain for $p=2$ and a Hilbert space $X$; the difficulty when
$p\ne 2$ and general $X$ for differential and pseudodifferential realizations is
linked to 
the fact that multiplier theorems valued in Hilbert spaces do not
in general
extend to Banach spaces. The difficulty was overcome by a deeper
analysis in \cite{GK93}, \cite{G95} for operators in the Boutet de
Monvel calculus (including differential boundary value problems and
nontrivial initial- and boundary conditions), in a
smooth setting. A nonsmooth case stemming from the Stokes problem was
treated in Abels \cite{A05}.

To include  nonsmooth settings in general, other tools have been introduced. We
shall in the present paper take advantage of the concept of $\mathcal
R$-boundedness, as developed 
 through works
of  Da Prato and Grisvard, Lamberton, Dore and Venni, Cl\'ement, Pr\"uss, Hieber,
Denk,
Weiss, Bourgain and others,
and explained very nicely in Denk-Hieber-Pr\"uss
\cite{DHP03}, which applies it to vector-valued nonsmooth differential operator
problems. The theory is also included in the book Pr\"uss-Simonett \cite{PS16}.

$\mathcal R$-boundedness of a family $\mathcal T$ of bounded linear operators
$T\colon X\to Y$ is defined as follows:

\begin{defn}\label{defn:4.1}
    Let $X$ and $Y$ be Banach spaces, and let
$\mathcal T$ be a family of operators $T$ in $ \mathcal L(X,Y)$. $\mathcal T$ is said
to be $\mathcal R$-bounded if there is a constant $C\ge 0$ and a $p\in
[1,\infty )$ such that  there holds: For each $N\in \N$,
$\{T_j\}_{j=1}^N\subset 
\mathcal T$, $\{x_j\}_{j=1}^N\subset X$, and $\{\varepsilon
_j\}_{j=1}^N$ belonging to a system of 
independent and identically distributed symmetric
$\{-1,+1\}$-valued random variables $\varepsilon $ on  some
probability space $(\Omega ,\mathcal M,\mu )$,
\begin{equation}
  \label{eq:4.3}
\|\sum_{j=1}^N\varepsilon _jT_jx_j\|_{L_p(\Omega ,Y)}\le C\|\sum_{j=1}^N\varepsilon _jx_j\|_{L_p(\Omega ,X)}.  
\end{equation}
\end{defn}
\begin{rem}
  \label{rem:4.1a}
  As the probability space and random variables,
one can for example take $(\Omega ,\mathcal M,\mu )=([0,1], \mathcal
B([0,1]),\lambda )$, where $\mathcal B([0,1])$ stands for the  Borel
$\sigma $-algebra, $\lambda $ for the Lebesgue measure, and the random
variables are given by the Rademacher functions, as explained in
detail  e.g.\ in
Denk \cite{D21}.
\end{rem}

An alternative formulation is given in Denk and Seiler
\cite{DS15}:

\begin{defn}\label{defn:4.2}
  Let $p\in [1,\infty )$. Denote $
Z_N=\{(z_1,\dots,z_N)\mid z_j\in \{-1,+1\}\text{ for all }j\}$, a
subset of $\R^N$.
Let $X$ and $Y$ be Banach spaces.

A subset $\mathcal T$ of the bounded
linear operators $\mathcal L(X,Y)$ is $\mathcal R$-bounded if there is a
constant $C\ge 0$ such that for every choice of $N\in \N$ and every
choice of $x_1,\dots,x_N$ in $X$ and $T_1,\dots,T_N$ in $\mathcal T$,
\begin{equation}
  \label{eq:4.4}
\bigl(\sum_{z\in Z_N}\|\sum _{j=1}^N z_jT_jx_j\|_Y^p\bigr)^{1/p}\le
C\bigl(\sum_{z\in Z_N}\|\sum _{j=1}^N z_jx_j\|_X^p\bigr)^{1/p}.
\end{equation}
\end{defn}

The finiteness for one $p\in [1,\infty
)$ implies the finiteness for all other  $p\in [1,\infty
)$.
The best constant $C$, denoted $\mathcal R_{\mathcal L(X,Y)}
(\mathcal T )$ or just $\mathcal R(\mathcal T)$, is called the
$\mathcal R$-bound of $\mathcal T$ (for some fixed $p$).
 An $\mathcal R$-bounded set is norm-bounded. Finite families $\mathcal
T$ are $\mathcal R$-bounded. Norm bounds and $\mathcal R$-bounds are
equivalent if $X$ and $Y$ are
Hilbert spaces. 

The $\mathcal R$-boundedness is preserved under addition and composition
(\cite[Prop.\ 3.4]{DHP03}):

\begin{prop}\label{prop:4.3}
    $1^\circ$ Let $X$ and $Y$ be Banach
spaces, and let $\mathcal T$ and $\mathcal S$ $\subset \mathcal L(X,Y)$ be $\mathcal
R$-bounded. Then 
$$
\mathcal T+\mathcal S=\{T+S\mid T\in\mathcal T, S\in \mathcal S\}
$$
is $\mathcal R$-bounded, and $\mathcal R\{\mathcal T+\mathcal S\}\le \mathcal R\{\mathcal
T\}+\mathcal R\{\mathcal S\} $.

 $2^\circ$ Let $X$, $Y$ and $Z$ be Banach
spaces, and let $\mathcal T\subset \mathcal L(X,Y)$ and $\mathcal S\subset \mathcal L(Y,Z)$
be $\mathcal
R$-bounded. Then 
$$
\mathcal S\mathcal T=\{ST\mid T\in\mathcal T, S\in \mathcal S\}
$$
is $\mathcal R$-bounded, and $\mathcal R\{\mathcal S\mathcal T\}\le \mathcal R\{\mathcal
S\}\mathcal R\{\mathcal T\} $.
\end{prop}

The fundamental interest of this concept is that it leads to a
criterion for maximal $L_p$-regularity, shown in \cite[Theorem~4.4]{DHP03}:

\begin{thm}\label{thm:4.4}
    Let  $1<p<\infty $ and $X$ be a UMD-space. Problem \eqref{eq:4.1} has maximal $L_p$-regularity
 on $I=\mathbb R_+$ if and only if $V_{\delta,0}\subset \rho(A)$ and the family $\{\lambda (
 A-\lambda )^{-1}\mid \lambda \in V_{\delta ,0}\}$ in $\mathcal L(X)$ is $\mathcal
R$-bounded  for some $\delta \in (0,\frac{\pi}2)$.
\end{thm}

Note that $\mathcal R$-boundedness of $\{\lambda (  A-\lambda )^{-1}\mid
\lambda \in V_{\delta ,K}\}$ implies that for some $k >K$,  $\mathcal R$-boundedness holds for $\{\lambda (  A+k -\lambda )^{-1}\mid
\lambda \in V_{\delta ,0}\}$. Then
the shifted operator $  A+k $ has maximal $L_q$-regularity on
$\mathbb R_+$, and $  A$ itself has it on finite intervals $I=(0,T)$.

We shall say that {\it $A$ is $\mathcal R$-sectorial on $V_{\delta ,K}$} when $V_{\delta ,K}\subset \rho(A)$ and
\begin{equation}
  \label{eq:4.5}
\mathcal R_{\mathcal L(X)}\{\lambda (A-\lambda )^{-1} \mid \lambda \in V_{\delta ,K} \}<\infty .
\end{equation}

One of the reasons that Theorem~\ref{thm:4.4} is particularly useful, is that $\mathcal
R$-sectoriality is preserved under suitable perturbations of $ A$.

\begin{prop}
  \label{prop:4.5}
    $1^\circ$ Let $A$  satisfy $V_{\delta,K}\subset \rho(A)$ and
  \begin{equation}
    \label{eq:4.6}
\|\lambda (  A-\lambda )^{-1}\|_{\mathcal L(X)}\le C\quad \text{for all }\lambda \in V_{\delta ,K}.
  \end{equation}
Let $S\colon D(  A)\to X$ be linear and satisfy 
\begin{equation}
  \label{eq:4.7}
\|Su\|_X\le \alpha \|  Au \|_X+\beta \|u\|_X\quad \text{for all }u\in D( 
A).
\end{equation}
Then when $\alpha$ is sufficiently small, there exists $K_1\ge K$ and $C'$
such that $V_{\delta,K_1}\subset \rho(A+S)$ and
$$
\|\lambda (  A+S-\lambda )^{-1}\|_{\mathcal L(X)}\le C'\quad\text{for all }\lambda \in V_{\delta ,K_1}.
$$

$2^\circ$ Assume in addition that $\{\lambda (  A-\lambda )^{-1}\mid
\lambda \in V_{\delta ,K}\}$ is $\mathcal R$-bounded. Then, for sufficiently
small $\alpha >0$ there is a $K_2\ge K$ such that $V_{\delta,K_2}\subset \rho(A+S)$ and   $\{\lambda (  A+S-\lambda )^{-1}\mid
\lambda \in V_{\delta ,K_2}\}$ is $\mathcal R$-bounded.
\end{prop}

\begin{proof} $1^\circ$ has been known for many years. A version is
proved in  \cite[Theorem~1.5]{DHP03}, which adapts straightforwardly  to our sectorial sets. 
$2^\circ$ is an adaptation of \cite[Prop.~4.3]{DHP03} in a similar way.
\end{proof}

We shall supply these results with some further properties that are
essential for our studies here:

\begin{lem}
  \label{lem:4.6}
  Let $X,Y,Y_0,Y_1$ be Banach spaces satisfying
$Y_1\subset Y\subset Y_0$, with continuous, dense injections. Assume that for some
$\theta \in (0,1)$, $C>0$,
$$
\|x\|_Y\le C\|x\|_{Y_0}^{\theta }\|x\|_{Y_1}^{1-\theta }\quad \text{for all
}x\in Y_1.
$$
Then for any operator family $\mathcal T$ in $\mathcal L(X,Y_0)\cap \mathcal L(X,Y_1)$,
the $\mathcal R$-bound of the operators considered as elements of $\mathcal
L(X,Y)$ satisfies
$$
\mathcal R_{\mathcal L(X,Y)}(\mathcal T)\le C \mathcal R_{\mathcal L(X,Y_0)}(\mathcal T)^{\theta }\mathcal R_{\mathcal L(X,Y_1)}(\mathcal T)^{1-\theta }.
$$
\end{lem}

\begin{proof}
Follows from the definition of $\mathcal R$-boundedness:

When $\Omega $, $\varepsilon _j$, $T_j\in \mathcal T$ and $x_j$ are as in Definition~\ref{defn:4.1},
\begin{align*}
\|\sum_{j=1}^N\varepsilon _jT_jx_j\|_{L_p(\Omega ,Y)}&\le C
\|\sum_{j=1}^N\varepsilon _jT_jx_j\|_{L_p(\Omega ,Y_0)}^{\theta }
\|\sum_{j=1}^N\varepsilon _jT_jx_j\|_{L_p(\Omega ,Y_1)}^{1-\theta }\\
&\le C \mathcal R_{\mathcal L(X,Y_0)}(\mathcal T)^{\theta }\mathcal R_{\mathcal
L(X,Y_1)}(\mathcal T)^{1-\theta }\|\sum_{j=1}^N\varepsilon
_jx_j\|_{L_p(\Omega ,X)}.
\end{align*}
 \end{proof}

 \begin{thm}
   \label{thm:4.7}
   Let $A$ be a closed, densely defined
linear operator in a Banach space $X$, such that $A$ is $\mathcal
R$-sectorial over $V_{\delta ,K}$. Let $Y$ be a Banach space
satisfying $D(A)\subset Y\subset X$ with dense, continuous
injections, and assume that for some $\theta \in [0,1]$ and $C_0>0$,
\begin{equation}
  \label{eq:4.8}
\|u\|_Y\le C_0\|u\|_{X}^{\theta }\|u\|_{D(A)}^{1-\theta}\quad \text{for all }u\in D(A).
\end{equation}
\begin{enumerate}
\item[$1^\circ$] With $$
C_1=\mathcal R_{\mathcal L(X)}\{\lambda (A-\lambda )^{-1}\mid
\lambda \in V_{\delta ,K}\},
$$ one has for any $K_1\ge K$ with $K_1>0$ that the $\mathcal R$-bound of $(A-\lambda )^{-1}$ over $V_{\delta
,K_1}$ satisfies
\begin{equation}
  \label{eq:4.9}
\mathcal R_{\mathcal L(X)}\{ (A-\lambda )^{-1}\mid
\lambda \in V_{\delta ,K_1}\}\le 2C_1/K_1. 
\end{equation}
\item[$2^\circ$] We have
\begin{equation}
  \label{eq:4.10}
\mathcal R_{\mathcal L(X,D(A))}\{(A-\lambda )^{-1}\mid
\lambda \in V_{\delta ,K}\}=C_2<\infty ,
\end{equation}
and when $S\in \mathcal L(D(A),X)$, then $\{S(A-\lambda )^{-1}\mid
\lambda \in V_{\delta ,K}\}$ is $\mathcal R_{\mathcal L(X)}$-bounded.
\item[$3^\circ$] Let $S\in \mathcal L(Y,X)$. With $\theta $ as in \eqref{eq:4.8}, there
is a constant $C$ such that for all $K_1\ge K$ with $K_1>0$,
$$
\mathcal R_{\mathcal L(X)}\{ S(A-\lambda )^{-1}\mid
\lambda \in V_{\delta ,K_1}\}\le CK_1^{-\theta }.
$$
\end{enumerate}
 \end{thm}
\begin{proof} $1^\circ$. Let $N\in \N$,
 $\{x_j\}_{j=1}^N\subset X$, and $\{\varepsilon
_j\}_{j=1}^N$ be as in Definition~\ref{defn:4.1}, $\{\lambda_j\}_{j=1}^N\subset V_{\delta,K_1}$ and $p\in [1,\infty)$. Then
\begin{align*}
K_1\|\sum_{j=1}^N \varepsilon_j(A-\lambda_j)^{-1}x_j\|_{L_p(\Omega;X)} 
&\leq 2\|\sum_{j=1}^N \varepsilon_j\lambda_j (A-\lambda_j)^{-1}x_j\|_{L_p(\Omega;X)}\\
&\leq 2C_1 \|\sum_{j=1}^N \varepsilon_jx_j\|_{L_p(\Omega;X)}
\end{align*}
by the contraction principle of Kahane (cf.\ e.g. Lemma~3.5 in \cite{DHP03}) since $|\lambda_j|\geq K_1$. This yields the first statement.

$2^\circ$. Since $A(A-\lambda )^{-1}=I-\lambda (A-\lambda
)^{-1}$, the $\mathcal R_{\mathcal L(X)}$-boundedness of the family  \linebreak$\{A(A-\lambda )^{-1}\mid
\lambda \in V_{\delta ,K}\}$ follows from that of   $\{\lambda (A-\lambda )^{-1}\mid
\lambda \in V_{\delta ,K}\}$; it moreover holds  for $(A-\lambda _0)(A-\lambda
)^{-1}$ for any $\lambda _0\in \C$. We here use  the sum rule
Proposition~\ref{prop:4.3} $1^\circ$.

Take $\lambda _0$ in the resolvent set of
$A$; then $A-\lambda _0$ is a homeomorphism of $D(A)$ onto $X$, so for
$(A-\lambda )^{-1}$ viewed as the composition of $(A-\lambda
_0)^{-1}\colon X\to D(A)$ and $(A-\lambda _0)(A-\lambda )^{-1}\colon
X\to X$, we get \eqref{eq:4.10} by the product rule  Proposition~\ref{prop:4.3} $2^\circ$.  

Moreover, we can write 
$$
S(A-\lambda )^{-1}=S(A-\lambda _0)^{-1}(A-\lambda _0)(A-\lambda
)^{-1}.
$$  Since
$S(A-\lambda _0)^{-1}\in \mathcal L(X)$, the last statement in $2^\circ$ follows from the
product rule.

$3^\circ$. Because of \eqref{eq:4.8}, we have by Lemma~\ref{lem:4.6},
\begin{align*}
&\mathcal R_{\mathcal L(X,Y)}\{ (A-\lambda )^{-1}\mid
\lambda \in V_{\delta ,K_1}\}\\
&\le C_0\mathcal R_{\mathcal L(X,X)}\{ (A-\lambda )^{-1}\mid
\lambda \in V_{\delta ,K_1}\}^{\theta }
\mathcal R_{\mathcal L(X,D(A))}\{ (A-\lambda )^{-1}\mid
\lambda \in V_{\delta ,K_1}\}^{1-\theta }\\
&\le C_0(C_1/K_1)^{\theta }C_2^{1-\theta }=C_3K_1^{-\theta },
\end{align*}
where we used \eqref{eq:4.9} and \eqref{eq:4.10}. If $S\in \mathcal L(Y,X)$ with norm $C_4$, we then find
by the product rule 
$$
\mathcal R_{\mathcal L(X)}\{ S(A-\lambda )^{-1}\mid
\lambda \in V_{\delta ,K_1}\}\le C_4C_3K_1^{-\theta }.
$$
\end{proof}

\begin{rem}\label{rem:4.8}
These general results will in the following be applied to the
situation where $A$ is the realization in $X=L_q(\Omega )$ of a pseudodifferential
operator $P$ satisfying Hypothesis~\ref{hypo:3.1}, with domain
$D(A)=D_q(\comega)$ \eqref{eq:3.6}, $\Omega $ being open, bounded and
$C^{1+\tau }$. The perturbation $S$ will often be taken as an operator
of order $s<a+1/q$, $s\ge 0$,
satisfying $\|Su\|_{L_q(\Omega )}\le c\|u\|_{\dot H^s_q(\comega )}$. Recall that
$H_q^{a(t)}(\comega)=\dot H_q^t(\comega)$ when $t<a+1/q$. Since
$s<a+1/q$, there is a $t$ with $s<t<a+1/q $, and there is an
interpolation inequality
\begin{equation}
  \label{eq:4.11}
\|u\|_{\dot H_q^{s}(\comega)}\le c\|u\|^\theta _{L_q(\Omega )}\|u\|^{1-\theta }_{\dot H_q^{t}(\comega)}\quad \text{for all }u\in \dot H_q^{t}(\comega)
\end{equation}
with a  $\theta \in (0,1)$ (more precisely, $\theta =1-s/t$, cf.\
Triebel \cite[1.3.3/5, 2.4.2]{T78}). Here $\|u\|_{\dot H_q^{t}(\comega)}=\|u\|_{
H_q^{a(t)}(\comega)}\le c \|u\|_{
H_q^{a(2a)}(\comega)}=c\|u\|_{D_q(\comega)}$. Thus 
\begin{equation}
  \label{eq:4.12}
\|u\|_{\dot H_q^{s}(\comega)}\le c'\|u\|^\theta _{L_q(\Omega
)}\|u\|^{1-\theta }_{D_q(\comega
)}\quad \text{for all }u\in D_q(\comega
) . 
\end{equation}
This also implies that for any $\varepsilon >0$ there is a constant
$C_\varepsilon $ such that
\begin{equation}
  \label{eq:4.13}
\|u\|_{\dot H_q^{s}(\comega)}\le \varepsilon \|u\|_{D_q(\comega)}+
C_\varepsilon \|u\| _{L_q(\Omega )}\quad \text{for all }u\in D_q(\comega
),
\end{equation}
showing that $S$ satisfies \eqref{eq:4.7} with arbitrarily small $\alpha>0 $.

If $s<a$, we can take $t=a$ in the interpolation.
\end{rem}

\section{Resolvent $\mathcal R$-bounds for the Dirichlet problem}\label{sec:5}

In the following we shall show how resolvent $\mathcal R$-bounds can be
obtained for a general class of $x$-dependent operators $P$ from the
knowledge in some special cases.

First consider pseudodifferential operators on $\rn$ without boundary
conditions. They can
be handled in a way based directly on symbolic calculus, as in
\cite{G95} and \cite{G18a} (when $\tau =\infty $).

\begin{prop}\label{prop:5.1}
   Let $P=\Op(p)$ with $p\in S^d(\rn\times\rn)$, homogeneous of order $d>0$ in $\xi $ satisfying
$\operatorname{Re}p(x,\xi )\ge c|\xi |^d$ for all $|\xi |\ge 1$, $x\in\rn$. Then for every $1<q<\infty $ and
a suitable constant $b$, the heat problem
\eqref{eq:4.1} for $A=P+b$ with $D(P+b)=H^d_q(\rn)$ and $X=L_q(\rn)$  has maximal $L_q$-regularity
on $\rp$, and hence  $P+b$ is $\mathcal R$-sectorial on $V_{\delta ,0}$
for some $\delta >0$.  
\end{prop}
\begin{proof} For integer $d$, this follows from Theorem~3.1 (1) in
\cite{G95}, where mapping properties in anisotropic Bessel-potential
spaces $H_q^{(s,s/d)}(\rn\times \R)$ were established; they hold for
$P+\partial_t$ as well as its parametrix, as formulated for the case with
boundary in Theorem~3.4 there. The property of being supported for $t\ge 0$ is
preserved by these mappings, since the symbols are holomorphic in $\tau $ for
$\operatorname{Im}\tau <0$. Here $b$ can be chosen so that $P+b$ has
positive lower bound in $L_2$ (by the G\aa{}rding inequality); then
there is a solution operator, which also works in the $L_q$-setting, and 
\begin{equation}
  \label{eq:5.1}
f\in L_q(\rn \times \rp)\iff u\in \dot H_q^{(d,1)}(\rn\times
\crp)=L_q(\rp;H^d_q(\rn))\cap \dot H^1_q(\crp;L_q(\rn)), 
\end{equation}
where $\dot H^1_q(\crp;L_q(\rn))=\{f\in H^1_q(\R_+;L_q(\rn))\mid f|_{t=0}=0\}$.
Noninteger $d$ are included in the detailed presentation of symbol
classes in \cite{G18a}. In that paper, the
emphasis is on the regularity conclusion $\implies$ in \eqref{eq:5.1}; the
existence is shown as in \cite{G95}. 
\end{proof}

Next, there is a special result for operators on a bounded domain.

\begin{prop}\label{prop:5.2}
Let $p^h\colon \rn\to \C$ be  smooth in $\rn\setminus \{0\}$,
strictly homogeneous of degree $2a>0$, even, strongly elliptic and real (so
$p^h(\xi )\ge c|\xi |^{2a}$ for all $\xi\in\Rn$ with $c>0$). Let $\ol p\colon \rn\to \C$ be smooth, coinciding with $p^h(\xi )$ for
$|\xi |\ge 1$ and positive for all $\xi\in\rn $. Let $P^h=\Op(p^h)$, $\ol P=\Op(\ol p)$.
Let $\Omega $ be a bounded
$C^{1+\tau }$-domain, $\tau >2a$.

Then there are  $\delta >0$ and
$K\ge 0$
such
that the $L_q$-Dirichlet realization $P^h_D$  of
$P^h$  on $\Omega $ is
 $\mathcal R$-sectorial on  $V_{\delta ,0}$, and the  $L_q$-Dirichlet
 realization $\ol P_D$  of  $\ol P$ on $\Omega $ is  $\mathcal
 R$-sectorial on $V_{\delta ,K}$.
\end{prop} 
\begin{proof}
The operator $P^h$ is of the kind $\mathcal L$ considered around \eqref{eq:3.2},
its $L_2(\Omega )$-Dirichlet realization being associated with the quadratic
form $Q$ recalled in \eqref{eq:3.8}. It is accounted for in \cite{G18a} around
(5.10) how the form $Q(u)$ is a so-called Dirichlet
form in the sense of Fukushima, Oshima and Takeda \cite{FOT94} (also
considered in Davies \cite{D89}).
It has
a Markovian property, which assures that $-P^h_D$ generates a strongly
continuous contraction semigroup $T_q(t)$ not only in $L_2(\Omega )$ but also
in $L_q(\Omega )$ for $1<q<\infty $, and $T_q(t)$ is bounded
holomorphic (and the operators for varying $q$ are consistent).
By Lamberton \cite{L87}, these properties imply that the heat problem
\eqref{eq:4.1} with $A=P^h_D$  has maximal $L_q$-regularity, for $1<q<\infty
$ and all finite intervals $I$.

It is also
shown in \cite{L87} that the constant $C$ in the estimates over $\Omega \times
I$, $I=(0,T)$,
\begin{equation}
  \label{eq:5.2}
\|Au\|_{L_q(\Omega \times I)}+\|\partial_tu\|_{L_q(\Omega \times I)}\le C\|f\|_{L_q(\Omega \times I)},
\end{equation}
is independent of $T$. This allows us to conclude that \eqref{eq:5.2} also holds
with $I=\rp$. \cite{L87} applies to very general, also unbounded,
sets $\Omega $, and what we have said so far, only shows that $A=P^h_D$ has
the weak maximal-regularity property defined in Pr\"uss-Simonett
\cite{PS16} p.\ 142 (is in  $ _0\mathcal M\mathcal R_q(\rp;L_q(\Omega ))$ in their notation).

Now since $\Omega $ is bounded, the quadratic form
$Q(u)$ on $\dot H^a(\comega )$ moreover satisfies a
Poincar\'e{} inequality (as accounted for in Ros-Oton \cite{R16}) so
that 0 is in the resolvent set of $P^h_D$. Then by Cor.\ 3.5.3 in \cite{PS16}, $P^h_D$
has the full maximal-regularity property (is in $\mathcal M\mathcal
R_q(\rp;L_q(\Omega ))$ in their notation). It means that
$\|u\|_{L_q(I,D(A))}$ can be added to the left-hand side in \eqref{eq:5.2}
with $I=\rp$.

We then conclude from Theorem~\ref{thm:4.4} that $P^h_D$ is $\mathcal R$-sectorial
on $V_{\delta ,0}$ for some $\delta >0$.

Since $\ol P-P^h$ has bounded symbol supported in $|\xi |\le 1$, it defines
a smoothing operator over $\Omega $. Then by Proposition~\ref{prop:4.5} $2^\circ$,  $\ol P_D$ is $\mathcal
R$-sectorial on $V_{\delta ,K}$ for some $K\ge 0$. 
\end{proof}

\begin{rem}\label{rem:5.3}
    It will also be used that $\mathcal
R$-sectoriality is preserved under suitable
coordinate transformations (such as those used in \cite{AG23}). This
holds, since composition with a single operator preserves $\mathcal
R$-boundedness (by Proposition~\ref{prop:4.3} $2^\circ$).
\end{rem}

 Denote the ball  $\{|x-x_0|<r\}$ in $\rn $ by $B_r(x_0)$; if $x_0=0$,
 we just write $B_r$. The closure is denoted $\ol B_r(x_0)$. The balls in $\R^{n-1}$ will be denoted
 $B'_r(x'_0)$, or just $B'_r$ if $x'_0=0$.
By $\chi _{r,s}$ ($r>s>0$) we denote a function in $C_0^\infty (\rn, [0,1])$ such that
$\supp \chi _{r,s}\subset B_r$ and $\chi _{r,s}(x)=1$ for $x\in
B_s$. Denote in particular
\begin{equation}
  \label{eq:5.3}
\chi _{2,1}=\eta ,\quad \chi _{1,\frac12}=\psi .
\end{equation}

 {The next result is the first crucial step in the regularity
estimates for bounded domains, taking place in a highly localized setting. The proof is
modeled after Theorem 6.6 in \cite{AG23}, but has the additional
features that $\mathcal R$-boundedness is taken into account, and the
comparisons over curved halfspaces in \cite{AG23} must here be
replaced by comparisons over truncated curved halfspaces, since the
point of departure is a result for bounded domains.}

We shall show:
{\begin{thm}\label{thm:5.4}
  Let $\Omega $ be bounded with  $C^{1+\tau
}$-boundary, $\tau >2a$, and let $1<q<\infty $. Let
$P=\operatorname{OP}(p)$ satisfy Hypothesis~\ref{hypo:3.1}. Assume moreover
that for $x\in\partial\Omega $, $p_0(x,\xi )$ is {\bf real} $>0$.
Let $P_0=\operatorname{OP}(p_0)$.

Consider a point $x_0\in \partial\Omega $, and denote $p_0(x_0,\xi )=\ol p(\xi )$ for all $\xi\in \rn$,  $\operatorname{OP}(p_0(x_0,\cdot ))=\ol P$.
Translate $x_0$ to $0$, and let $U$ be a neighborhood of
$0$ where, after a rotation,  $U\cap \Omega $
has the form $U \cap \rn_{\zeta _1} $ for a function $\zeta _1\in C^{1+\tau
}(\R^{n-1},\R)$ with $\zeta _1 (0)=0$, $\nabla\zeta _1(0)=0$. By a
dilation we can assume that $U$ contains $B'_2\times [-M,M]$, where $M=\max _{|x'|\le 2}\{ |\zeta _1(x')|,2\}$.

Then there exists a $z\in (0,1]$ such that the following holds: There is a bounded
$C^{1+\tau }$-domain $\Sigma _1$ with $B_z\cap \Omega =B_z\cap
\Sigma _1$, and an operator $P_1$ satisfying Hypothesis~\ref{hypo:3.1} such that for $u\in D_q(\comega)$ supported in $B_{z/4}$,
\begin{equation}
  \label{eq:5.4}
\chi _{ {z,2z}} (P_0-\lambda )u=\chi _{ {(1+\varepsilon
)z,z}}(P_1-\lambda )u \text{ on }\rn
\end{equation}
(some $\varepsilon >0$), where $P_{1,D}\colon D_q(\ol\Sigma _1)\to L_q(\ol\Sigma _1)$ is $\mathcal R$-sectorial on $ V_{\delta ,K}$ for some
$K\ge 0$. Consequently, for any $\varphi  \in C_0^\infty (B_{z/2})$
\begin{equation}
  \label{eq:5.5}
\varphi  (P_0-\lambda )u=\varphi (P_1-\lambda )u\text{ on }\rn. 
\end{equation}
\end{thm}

\begin{proof} 
 Departing from Proposition~\ref{prop:5.2}, we will show the formula by use of a scaling argument, making it possible
to find a small set where  $P_0-\ol P$ and
$\zeta _1(x')$ have so small values that the resolvent estimates for
$\ol P$ can
be carried over to $P_0$.  {To perform the scaling argument more easily we translated $x_0$ to $0$.}

\medskip

\noindent
 {\emph{Step 1 (Small perturbations of constant coefficients and flat domains):} We} introduce an auxiliary domain:
Along with $\zeta _1\in C^{1+\tau }(\R^{n-1},\R)$, consider $\zeta
(x')=\chi _{2,1}(x')\zeta _1(x')$ for all $x'\in\R^{n-1}$, coinciding with $\zeta _1 $ when
$|x'|\le 1$ but vanishing for $|x '|\ge 2$. We now choose a $C^{1+\tau }$ set
$\Sigma $ such that for $|x'|\ge 2$, it is a subset of the slab
$\{x\in \rn\mid 0<x_n<2M\}$ containing the cylindrical set $\{x\mid 2\le
|x'|\le 5,\; 0<x_n< 2M\}$, and for $|x '|\le 2$ it is the set $V=\{x\in \rn\mid
|x'|\le 2,\; \zeta (x')<x_n<2M\}$.

The diffeomorphism $F_{\zeta }\colon (x',x_n)\mapsto (x',x_n-\zeta
(x'))$ sends $\rn_{\zeta }$ bijectively to $\rnp$; it acts as  the identity on points
outside the cylinder $B'_2\times \R$, and maps $V$ to a set $V'\subset B'_2\times [0,2M]$, which
has the boundary piece $B'_2\times \{0\}$ in common with $\rnp
$. Denote $F_\zeta (\Sigma )=\Sigma ' $.
Recall from \cite[Section~6]{AG23} that under the diffeomorphism
$F_\zeta $ on $\rn$, a suitable
operator $P$, acting on functions defined on $\Sigma$ or $\R^n$, is carried over to the operator $P_\zeta =F^{-1,\ast}_\zeta
PF_\zeta^\ast $. 

Now a slight variant of \cite[Proposition~6.5]{AG23} is needed:
\begin{lem}\label{lem:5.5}
   Let  $\ol{p}\in S^{2a} (\rn\times
\rn)$, $p\in C^\tau S ^{2a}(\rn\times
\rn)$, and $1<q<\infty$.
For any $\varepsilon >0$ 
there exist $k\in\N$ and $\varepsilon ' =\varepsilon '(\ol{p},q)>0$ such that if 
\begin{equation}
  \label{eq:5.6}
  |\ol{p}-p|_{k,C^\tau S^{2a} _{1,0}(\rn\times \rn)}\leq \varepsilon '\quad
  \text{and} \quad \|\zeta\|_{C^{1+\tau    }(\R^{n-1})}\leq
  \varepsilon '.    
\end{equation}
then $\|\ol P- P_\zeta \|_{\mathcal
L(D _q(\ol \Sigma ' ),L_q(\Sigma '  ))}\le \varepsilon $.
\end{lem}
 {The proof of this lemma is given below.} 

We continue the proof of  {Theorem}~\ref{thm:5.4}, with $P, \ol{P}, P_0, p,
\ol{p}$ defined there. Here we have that Proposition~\ref{prop:5.2} applies to $\ol P$
considered over $\Sigma ' $. Then Proposition~\ref{prop:4.5} can be applied to
 $P_{0,\zeta }$ as a perturbation of $\ol P$, when the norm difference is small enough,
and by Lemma~\ref{lem:5.5},
this can be obtained when the symbol
estimates of $\ol p-p_0$ and $C^{1+\tau }$-estimate of $\zeta $ in
\eqref{eq:5.6} are small enough. Thus we get for such $p_0$ and $\zeta $ close to $\overline p$ and $0$, resp.:
\begin{equation}
  \label{eq:5.7}
\|(P_{0,\zeta }-\lambda )v\|_{L_q(\Sigma ' ) }\ge c_0|\lambda |\|v\|_{L_q (\Sigma ' )}\quad \text{for all }\lambda
\in V_{\delta ,K_0},\; v\in D _q(\ol\Sigma ' ), 
\end{equation}
and moreover, the family $\lambda (P_{0,\zeta ,D}-\lambda )^{-1}$ is $\mathcal R$-bounded for
$\lambda \in V_{\delta ,K_0}$ (here $P_{0,\zeta ,D}$ stands for the
Dirichlet realization of $P_{0,\zeta }$ on $\Sigma '$).

For such $p_0$ and $\zeta $,  a
similar estimate can be concluded for $P_0$ over $\Sigma $, by changing variables back to
$\Sigma $: 
\begin{equation}
  \label{eq:5.8}
\|(P_{0 }-\lambda )v\|_{L_q(\Sigma  ) }\ge c_0|\lambda |\|v\|_{L_q (\Sigma  )}\quad \text{for all }\lambda
\in V_{\delta ,K_0},\; v\in D _q(\ol\Sigma  ); 
\end{equation}
also $\mathcal R$-boundedness is preserved here, cf.\ Remark~\ref{rem:5.3}.

\medskip

\noindent
 {\emph{Step 2 (Local scaling):}}  {We will now use a scaling argument to reduce the statement for a general operator $P_0$ to the case considered in the first step, i.e., an operator with a symbol close to a constant coefficient operator $\ol P$}, when applied to functions supported in a
sufficiently small ball around $0$.

Recalling that $\eta =\chi _{2,1}$, we define for $z>0$:
\begin{equation}
  \label{eq:5.9}
  \begin{split}
  \zeta_z(x')&= z^{-1}\eta ((x',0)) \zeta(zx'), \\
  p_z(x,\xi) &= \eta (x) p_0(zx,\xi)+ (1-\eta(x)) p_0(0,\xi),
   \end{split}
\end{equation}
for all $x,\xi\in\R^n$, $x'=(x_1,\ldots,x_{n-1})$.
Define moreover
\begin{equation}
  \label{eq:5.10}
  q_z(x,\xi)= p_0(zx,z^{-1}\xi)- z^{-2a}p_0(zx,\xi).  
\end{equation}
Because of the homogeneity of $p_0$,  $p_0(zx,z^{-1}\xi)=z^{-2a} p_0(zx,\xi)$ for all $|\xi|\geq 1$ and $z\in (0,1]$ and therefore
$q_z(x,\xi)=0$ for all $|\xi|\geq 1$, $z\in (0,1]$. Hence $q_z\in
C^\tau S^{-\infty}_{1,0}(\rn\times \rn)$.

For $v\colon \rn\to \C$ and $z>0$ we shall write $\sigma_z v\colon
\rn\to\C$ for the function $
  (\sigma_z v)(x)= v(zx)$.
We have with $P_0=\Op (p_0)$  
for all $x\in\R^n$ and all suitable $v\colon \rn\to\C$:
\begin{equation}
  \label{eq:5.11}
  \sigma_z(P_0v)(x)= \int_{\R^n} e^{izx\cdot \xi} p_0(zx,\xi) \hat{v}(\xi)\dd \xi
          = \int_{\R^n} e^{ix\cdot \xi} p_0(zx,z^{-1}\xi)
          \widehat{\sigma_z(v)}(\xi)\,\dd \xi,  
\end{equation}
Then, by use of $q_z$
\begin{equation}
  \label{eq:5.12}
 \sigma_z(P_0v)(x)= z^{-2a}(\Op (p_z) \sigma_z(v))(x) +
 (\Op(q_z)\sigma_z(v))(x)\quad\text{for all }|x|\le 3.    
\end{equation}

Denote $\operatorname{OP}(p_z)=P_z$, $\operatorname{OP}(q_z)=Q_z$, so
that \eqref{eq:5.12} reads
\begin{equation}
  \label{eq:5.13}
 \sigma_z(P_0v)(x)= z^{-2a}(P_z \sigma_z(v))(x) +
 (Q_z\sigma_z(v))(x)\;\quad \text{for all }|x|\le 3. 
\end{equation}

Recall from \cite[Lemma 6.7]{AG23}
the technical lemma that serves to control remainder terms:
\begin{lem}\label{lem:5.6}
For any $k\in\N$ there is some $C>0$ such that for all $z\in (0,1]$
\begin{equation}
  \label{eq:5.16}
  \|\zeta_z\|_{C^{1+\tau}(\R^{n-1})}\leq Cz^{\min (1,\tau) },\qquad
  |p_z-\ol{p}|_{k,C^{\tau}S^{2a}_{1,0}(\rn\times \rn)}\leq Cz^{\min
    (1,\tau) }. 
\end{equation}
\end{lem}

 Define  $\Sigma _z$
to be like $\Sigma $ for $|x'|\ge 2$, and for $|x'|\le 2$ to be of the form
$\{x\mid |x'|\le 2, \zeta _z(x')<x_n<2M\}$. Using Lemma~\ref{lem:5.6}, we can apply the same argumentation as around \eqref{eq:5.7}
to the difference $\ol P-P_{z,\zeta _z}$ to show  that  for a
sufficiently small $z\in (0,1]$,
$P_{z,\zeta _z}$ has an estimate
$$
\|(P_{z,\zeta _z}-\lambda )v\|_{L_q(\Sigma '_z)}\ge c'_1|\lambda |\|v\|_{L_q (\Sigma '_z)}\quad \text{for all }\lambda
\in V_{\delta ,K_1},\;v\in D_q(\ol\Sigma _z).
$$
This is carried back by diffeomorphism to show that
$P_z$ has an estimate
\begin{equation}
  \label{eq:5.17}
\|(P_z-\lambda )v\|_{L_q(\Sigma _z)}\ge c_1|\lambda |\|v\|_{L_q (\Sigma _z)}\quad \text{for all }\lambda
\in V_{\delta ,K_1},\;v\in D_q(\ol\Sigma _z). 
\end{equation}
Here the family $\lambda (P_{z,\zeta _z,D}-\lambda )^{-1}$ is $\mathcal
R$-bounded for $\lambda \in V_{\delta ,K_1}$, and then so is the family $\lambda (P_{z,D}-\lambda )^{-1}$.

We \emph{fix} such a $z$ in the following!

Note that when $v$ is supported in $B_2$, then $\|v\|_{D_q(\ol \Sigma
_{z})}$ identifies with $\|v\|_{D_q(\ol {\R}^n
_{\zeta _z})}$.
For functions $u$ supported in
$B_1$,  $\|u\|_{D_q(\comega)}$ can be replaced by
$\|u\|_{D_q(\ol{\Sigma} )}$, since  $D_q(\comega)$  is defined here by
the localization using $\zeta _1(x')$, which equals $\zeta (x')$ for
$|x'|\le 1$ (cf.\ the definition of the transmission space by local coordinates).

Now we consider a function $u\in D_q(\ol {\Sigma})$ with
support in $B_{z/4}$. By the definition of $\zeta _z$, the
function $\sigma _zu$ is in  $D_q(\ol {\Sigma}_z)$, supported in
$B_{1/4}$. We insert $u$ in \eqref{eq:5.13}, replace $P_0$ by
$P_0-\lambda $ by subtracting $\lambda  \sigma _zu$ from both sides, and multiply the resulting equation by
 $\psi =\chi _{1,1/2}$, so that the validity extends to $x\in \rn$.
We shall moreover
multiply the equation by $\psi '=\chi _{(1+\varepsilon ),1}$ for a
small $\varepsilon >0$; it satisfies $\psi '\psi =\psi $. This gives, since $\psi \sigma _zu=\sigma _zu$,
\begin{align*}
\psi(x) &\sigma_z((P_0-\lambda )u)(x)
 = z^{-2a}\psi(x) (P_z \sigma_z(u))(x) -\lambda \psi(x)  \sigma _zu(x)+
 \psi(x) (Q_z\sigma_z(u))(x)\\
 & = \psi '(x)[z^{-2a}\psi(x) (P_z \sigma_z(u))(x) -\lambda  \sigma _zu(x)+
 \psi(x) (Q_z\sigma_z(u))(x)]\quad\text{for all }x\in \rn.
\end{align*}
Here we can moreover use that 
$$
\psi P_z\sigma _zu=P_z(\psi \sigma _zu)+[\psi ,P_z]\sigma
_zu=P_z(\sigma _zu)+[\psi ,P_z]\sigma _zu,
$$
so that we get
\begin{equation}
  \label{eq:5.18}
 \psi \sigma_z((P_0
-\lambda )u)
 = \psi '[z^{-2a} P_z \sigma_zu -\lambda  \sigma _zu+
 \psi Q_z\sigma_zu+ 
 z^{-2a}[\psi ,P_z]\sigma _zu].
\end{equation}

Denote $S_z= \psi Q_z+ z^{-2a}[\psi ,P_z]$ ---  it is bounded from $
H_q^{\max\{0,2a-1\}}(\rn)$ to $L_q(\rn)$ --- then \eqref{eq:5.18} takes the
form
\begin{equation}
  \label{eq:5.19}
 \psi \sigma_z((P_0
-\lambda )u) = \psi '(z^{-2a} P_z +S_z-\lambda )\sigma_zu\quad \text{on }\rn.
\end{equation}
When $a\le \frac12$, $S_z$ is bounded in $L_q(\rn)$, and when
$a>\frac12$, it satisfies an inequality \eqref{eq:4.13} since $0<2a-1<a$. Then
we can apply Proposition~\ref{prop:4.5} with $A=z^{-2a}P_z$  and $S=S_z$ over
$\Sigma _z$, finding that the Dirichlet problem for  $z^{-2a}P_z+S_z$ over $\Sigma _z$
 has the desired type of estimate for some $K_2$ sufficiently large: 
 \begin{equation}
   \label{eq:5.20}
\|(z^{-2a} P_z + S_z-\lambda ) v\|_{L_q(\Sigma _z)}\ge c_2|\lambda
 |\|v\|_{L_q (\Sigma _z)}\quad \text{for }\lambda
\in V_{\delta ,K_2},
 \end{equation}
for all $v\in D_q(\ol\Sigma _z)$, with $\mathcal R$-boundedness of the family
$\lambda ((z^{-2a} P_z + S_z)_D-\lambda )^{-1}\in \mathcal L(L_q(\Sigma _z))$  for $\lambda
\in V_{\delta ,K_2}$.

\medskip

\noindent
 {\emph{Step 3 (Scaling back):}}
Finally, this will be scaled back to a replacement of $\sigma _zu$
(recall that it is short for $x\mapsto u(zx)$) by $u$. The set $\Sigma _z$ will then be
replaced by a set $\Sigma _1=\{y\in\rn\mid y/z\in \Sigma _z\}$, where
the important observation is that the piece where $|x'|<z$,
$\zeta _z(x')<x_n<2M$, is carried over to the piece where $|y'|<1$, $
\zeta (y')<y_n<2zM$, which coincides with a piece of $\Omega $. The
operator $z^{-2z}P_z+S_z$ is (by a formula as in \eqref{eq:5.11}) carried over to an operator we shall call $P_1$
(by a slight abuse of notation); 
$$
(z^{-2z}P_z+S_z)\sigma _zv= P_1v,
$$
 and $P_{1,D}$ now has the appropriate $\mathcal
R$-sectoriality over $\Sigma _1$. Note also that $\psi =\chi _{1,1/2}$ carries over to
 {$\chi _{z,z/2}=\sigma_z^{-1}(\chi_{1,1/2})$}. Formula \eqref{eq:5.19} then takes
the form 
$$
\chi _{1/z,1/2z}(P_0-\lambda )u=\chi _{(1+\varepsilon
)/z,1/z}(P_1-\lambda )u,\text{ when
}u\in D_q(\comega)\text{ with }\operatorname{supp}u\subset B_{z/4},
$$
showing \eqref{eq:5.4}. Multiplication by $\varphi $ on both sides gives
\eqref{eq:5.5}, ending the proof of  {Theorem}~\ref{thm:5.4}.
\end{proof}

\noindent
 {\begin{proof*}{of Lemma~\ref{lem:5.5}}
Proposition~6.5 in \cite{AG23} shows this with $\Sigma ' $ replaced by
$\rnp$ (the difficult part is the change of variables, prepared there
in Theorem~5.13). We note that in the latter proposition it is assumed that $p$ is strongly elliptic and even. But for the estimate in \eqref{eq:5.6} this is not needed. To obtain the  statement in the lemma, we decompose a function
$u\in D _q(\ol\Sigma ' )$, by use of fixed
smooth cut-off functions, into three terms $u=u_1+u_2+u_3$, with 
$\operatorname{supp}
u_1\subset B'_4\times [0,\frac32 M]$, $\operatorname{supp}
u_2\subset B'_4\times [M,3 M]$, and
$\operatorname{supp}u_3\subset \ol \Sigma ' \setminus ( B'_3\times
[0,2M])$; 
all three belonging to
$D _q(\ol\Sigma ' )$. The term $u_1$ can also be viewed as an
element of $D _q(\crnp)$, and the rule in \cite[Proposition~6.5]{AG23} pertaining to $\rnp $ and $\rn_{\zeta }$ applies. This yields $\|(\ol P- P_\zeta)u_1 \|_{L_q(\rn )}\leq \eps \|u_1\|_{D _q(\ol \Sigma ' )}$ if $\eps'$ is sufficiently small. For the
term $u_2$ there is a similar rule pertaining to the halfspace
$\{x\mid x_n<2M\}$ and
the curved halfspace $\{x\mid x_n+\zeta (x')<2M\}$. For $u_3$ there is a simpler rule since
the variable $x$ is not shifted. The norm
$\|(\ol P- P_\zeta)u_3 \|_{L_q(\rn )}$ will then be dominated by the norms in \eqref{eq:5.6} (times $\|u_3\|_{D _q(\ol \Sigma ' )}$), and
so will, a fortiori, the norm $\|\ol P- P_\zeta \|_{\mathcal
L(D _q(\ol \Sigma ' ),L_q(\Sigma '  ))}$.
This shows the lemma. 
\end{proof*}
}

There is  a  related, slightly easier statement for interior points:
\begin{prop}\label{prop:5.7}
    Let $P$, $P_0$ and $\Omega $ be as in  {Theorem}~\ref{thm:5.4}. Consider an interior point $x_0\in\Omega $.

Then there exists a $z\in (0,1]$ and a
 $P_1$ satisfying Hypothesis~\ref{hypo:3.1} such that the following holds: For $u\in
 D_q(\comega)$ supported in $B_{z/4}(x_0)$, and  $\varphi  \in C_0^\infty (B_{z/2}(x_0))$, we have 
 \begin{equation}
   \label{eq:5.20a}
\varphi  (P_0-\lambda )u=\varphi (P_1-\lambda )u\text{ on }\rn, 
 \end{equation}
where $P_{1}\colon H^{2a}_q(\rn)\to L_q(\rn)$ is $\mathcal R$-sectorial on $V_{\delta ,K}$ for some
$K\ge 0$. 
\end{prop}
\begin{proof} 
Here  we depart from Proposition~\ref{prop:5.1} in a similar way. Consider an
interior point
$x_0\in\Omega $. We have from Proposition~\ref{prop:5.1} that there are
 $\delta >0$ and $K\ge 0$ such that  $\ol P=\Op(p(x_0,\cdot ))$
satisfies an estimate
$$
\|(\ol P-\lambda )u\|_{L_q(\rn)}\ge c_0|\lambda |\|u\|_{L_q(\rn)}
\quad\text{for all }\lambda \in V_{\delta ,K},
$$
with $\mathcal R$-boundedness of $\lambda (\ol P-\lambda )^{-1}$ on $V_{\delta ,K}$.
By a dilation, we can assume that $\ol B_4(x_0)\subset \Omega $. There
is a version of Lemma~\ref{lem:5.5} stating that  {for every $\eps>0$ there is some $\eps'>0$ such that the first inequality in \eqref{eq:5.6}
assures that $\|\ol P- P_0 \|_{\mathcal
L(H^{2a} _q(\rn ),L_q(\rn  ))}\le \varepsilon $}. Then we get when $p_0$ is close enough to $\ol p$
that for some $K'\ge K$,
$$
\|( P_0-\lambda )u\|_{L_q(\rn)}\ge c_0|\lambda |\|u\|_{L_q(\rn)}
\quad \text{for all }\lambda \in V_{\delta ,K'},
$$
with $\mathcal R$-boundedness of $\lambda ( P_0-\lambda )^{-1}$ on $V_{\delta ,K'}$.

Define $p_z$ and $q_z$ as in \eqref{eq:5.9}ff. Solutions
supported in balls $B_r(x_0)$ with $r<4$ are then simply in $\dot
H_q^{2a}(\ol B_r(x_0))$ (and no modification of a boundary is
needed). The result is now obtained by repeating the arguments from
the proof of  {Theorem}~\ref{thm:5.4}, with $\rn$ as the auxiliary domain instead of
$\Sigma _z$. 
\end{proof}

Our aim is now  to use these very local statements to control  operators over $\Omega $.

It was shown in
\cite{G23} that the spectrum of the Dirichlet realization of $P$,
known in the $L_2$-setting to be a discrete set $\Sigma $ contained in a
sector opening to the right, is the same in the $L_q$-setting for all
$1<q<\infty $.  So we know already that the resolvent equation
\begin{equation}
  \label{eq:5.21}
(P-\lambda )u=f \text{ in }\Omega ,\quad u=0 \text{ in }\rn\setminus \Omega, 
\end{equation}
has a unique solution for $\lambda $ in a suitable sector $V_{\delta
,K}$; it is the estimate of the solution operator for large $\lambda $
that we need to
show.

Resolvent estimates are easy to deduce in the
$L_2$-setting from the variational theory. We want to obtain them
for general $q$, including $\mathcal R$-boundedness, when $P$ has real positive principal symbol at the
boundary points. 

\begin{thm}\label{thm:5.8}
Let $\Omega $ be bounded with  $C^{1+\tau
}$-boundary, $\tau >2a$, and let $1<q<\infty $. Let
$P=\operatorname{OP}(p)$ satisfy Hypothesis~\ref{hypo:3.1}, and
assume that the principal symbol $p_0(x,\xi )$ is real positive at
each boundary point   $x\in \partial\Omega$.
Then there are constants $\delta >0$, $c_0>0$ and $K_0\ge 0$ such that
$P-\lambda $ satisfies an estimate for all $u\in D_q(\comega)=H^{a(2a)}(\comega)$:
\begin{equation}
  \label{eq:5.22}
\|(P-\lambda )u\|_{L_q(\Omega   )}\ge
c_0|\lambda |\|u\|_{L_q ( \Omega  )}\text{ when }\lambda 
\in V_{\delta ,K_0},
\end{equation}
with $\mathcal R$-boundedness of the family $\{\lambda (P_D-\lambda
)^{-1}\mid \lambda \in V_{\delta ,K_0}\}$ in $\mathcal L(L_q(\Omega ))$.
\end{thm}

 \begin{proof}
 We can assume  $P=P_0$, since $P-P_0$ is a $\psi $do of order
$2a-1$ to which Proposition~\ref{prop:4.5}  can be applied as soon as the
estimates are established for $P_0$. (One here uses  Remark~\ref{rem:4.8},
observing that  $s=(2a-1)_+=\max\{2a-1,0\}$ is $<a$.)

By  {Theorem}~\ref{thm:5.4}, there is for every $x\in \partial\Omega $ a ball
$B_r(x)$ and an auxiliary $C^{1+\tau }$-domain $\Sigma _1$ and $\mathcal
R$-sectorial operator $P_1$ on $D_q(\ol\Sigma _1)$ such that $\varphi  (P_0-\lambda )u=\varphi 
(P_1-\lambda )u$ when $u\in D_q(\comega)$ with support in $B_{r/4}(x)$
and $\varphi  \in C_0^\infty (B_{r/2}(x))$; here the $B_s(x)$, $0<s\le r$, are
neighborhoods of the kind $U_j$ ($j\ge 1$) described before (2.15). A related statement holds for
interior points $x$, by Proposition~\ref{prop:5.7}; here the auxiliary domain
$\Sigma _1$ is simply $\rn$.
Since $\comega$ is compact, there is a finite cover $B_{r_i/4}(x_i)$, $i=1,\dots,N$, of
$\comega $ by such balls. Introduce a partition of unity $\{\varrho
_i\}_{i=1,\dots ,N}$ (with $ \varrho _i\in
C_0^\infty ( B_{r_i/4}(x_i), [0,1])$, satisfying $\sum_{1\le i\le N}\varrho _i=1$ on
$\comega$), and choose functions $\psi _i\in C_0^\infty (
B_{r_i/2}(x_i)$ that are 1 on $ B_{r_i/4}(x_i)$. Denote by $P_i$ and
$\Sigma _i$ the associated operator and domain for which
\begin{equation}
  \label{eq:5.24}
  \varphi  (P_0-\lambda )u=\varphi  (P_i-\lambda)u
\end{equation}
holds when $u\in D_q(\comega)$ with $\supp u\subset  B_{r_i/4}(x_i)$, and $\supp \varphi \subset
B_{r_i/2}(x_i)$, $\lambda \in V_{\delta ,K_i}$, according to
 {Theorem}~\ref{thm:5.4} and Proposition
\ref{prop:5.7}.

We want to construct an approximate inverse of $P_{0,D}-\lambda $ by use of
these identities in the local coordinate patches.

For a given  $f\in L_q(\Omega )$, let $u=u(\lambda )\in D_q(\comega)$
be the family of functions satisfying 
$$
(P_0-\lambda )u(\lambda )=f\text{ on }\Omega ,\text{ for } \lambda \in V_{\delta ,K} .
$$
By multiplication by $\varrho _i$, we find $\varrho _i(P_0-\lambda
)u=\varrho _if$, and hence
\begin{equation}
  \label{eq:5.25}
(P_0-\lambda )\varrho _iu=\varrho _if+[\varrho _i,P_0]u\quad \text{on }\Omega .
\end{equation}
Multiplication by $\psi _i$ gives 
$$
\psi _i(P_0-\lambda )\varrho _iu=\psi _i\varrho _if+\psi _i[\varrho
_i,P_0]u
=\varrho _if+\psi _i[\varrho
_i,P_0]u\quad \text{on }\Omega .
$$

By  {Theorem} \ref{thm:5.4} and Proposition \ref{prop:5.7}, the left-hand side equals $\psi
_i(P_i-\lambda )\varrho _iu$, hence
$$
\psi _i(P_iu-\lambda )\varrho _iu
=\varrho _if+\psi _i[\varrho
_i,P_0]u\quad \text{on }\Omega , \text{ supported in }B_{r_i/2}(x_i).
$$
In particular,
\begin{equation}
  \label{eq:5.26}
1_{\Sigma _i}\psi _i(P_i-\lambda )\varrho _iu
=1_{\Sigma _i}(\varrho _if+\psi _i[\varrho
_i,P_0]u)\quad\text{on }\Omega .
\end{equation}
Here we observe that 
$$
1_{\Sigma _i}\varrho _if=1_\Omega \varrho _if,
$$
when $f$ is considered as extended by 0 outside $\Omega $.

For the left-hand side of \eqref{eq:5.26}, we note that by a commutation with $\psi _i$,
$$
\psi _i(P_i-\lambda )\varrho _iu=(P_i-\lambda )\varrho _iu-[P_i,\psi
_i]\varrho _i\psi _iu, 
$$
since $\psi _i\varrho _i=\varrho _i$, and for the right-hand side,
$$
\psi _i[\varrho _i,P_0]u=[\varrho _i,P_0]\psi _iu+[\psi _i,[\varrho _i,P_0]]u;
$$
this leads to the formula 
\begin{equation}
  \label{eq:5.27}
\begin{split}
1_{\Sigma _i}(P_i-\lambda
)\varrho _iu
&=1_{\Sigma _i}\varrho _if+1_{\Sigma _i}S_i\psi _iu+1_{\Sigma
_i}S_i'u, \text{ with}\\
S_i&=[P_i,\psi _i]\varrho _i+[\varrho _i,P_0],\\
S_i'&=[\psi _i,[\varrho _i,P_0]].
\end{split}
\end{equation}
Here $S_i$ is a $\psi $do of order $2a-1$ and $S_i'$ is of order
$2a-2$; the latter order is $\le 0$ and the former is so when $a\le \frac12$.

Now compose all this with
$(P_{i,D}-\lambda )^{-1}\colon L_q(\Sigma _i)\to D_q(\ol\Sigma _i)$,
arriving at
\begin{equation}
  \label{eq:5.28}
\varrho _iu
=(P_{i,D}-\lambda )^{-1}1_{\Sigma _i}\varrho _if+(P_{i,D}-\lambda )^{-1}1_{\Sigma _i}(S_i\psi _iu+S_i'u). 
\end{equation}
This has the form of an $\mathcal R$- {bounded} operator family acting on $f$ and two
operators acting on $u$ with lower
order factors, one of them applied to the global $u$. Summation over $i$ gives a
representation of $u=R_\lambda f$ as an $\mathcal R$- {bounded} sum  and a
remainder term that should behave better for $|\lambda |\to\infty $:
\begin{equation}
  \label{eq:5.29}
  \begin{split}
R_\lambda f&=u=R_{0,\lambda }f+T_\lambda u, 
\text{ where }\\
R_{0,\lambda }f&=\sum_{i\le N}(P_{i,D}-\lambda )^{-1}1_{\Sigma
_i}\varrho _if,\\
T_\lambda u&=\sum_{i\le N}(P_{i,D}-\lambda )^{-1}1_{\Sigma _i}(S_i\psi
_iu+S_i'u). 
\end{split}
\end{equation}
Here we let $\lambda \in V_{\delta ,\ol K}$, where $\ol K=\max_{i\le
N}\{K_i\}$. The first line shows:
\begin{equation}
  \label{eq:5.30}
(1-T_\lambda )R_\lambda =R_{0,\lambda }\quad \text{on $\Omega $ for }\lambda \in
V_{\delta ,\ol K}
.
\end{equation}

To obtain a useful formula for $R_\lambda $ from $R_{0,\lambda }$ and
$T_\lambda $ is easiest when $a\le \frac12$, since all the $S_i$ and  $S_i'$
are then bounded in $L_q$-norm. However, we shall give just one formulation of
the proof that works for all $0<a<1$.

Consider
\begin{equation}
  \label{eq:5.33}
H_\lambda 
=\sum _{k=0}
^\infty T_\lambda ^kR_{0,\lambda }.
\end{equation}
If the series converges, then 
$$
H_\lambda -T_\lambda H_\lambda= \sum _{k=0}
^\infty T_\lambda ^kR_{0,\lambda }-\sum _{k=1}
^\infty T_\lambda ^kR_{0,\lambda }=R_{0,\lambda },
$$
so 
$$
(1-T_\lambda )H_\lambda =R_{0,\lambda }.
$$
This is the equation, $R_\lambda $ should solve, cf.\ \eqref{eq:5.30}. If
$1-T_\lambda $ is invertible in a suitable sense,
we can conclude that
$R_\lambda =H_\lambda $.

Let us first investigate the invertibility of $1-T_\lambda $. We have for $R_{i,\lambda }=(P_{i,D}-\lambda )^{-1}$ the standard resolvent
estimates when $\lambda \in V_{\delta ,\ol K}$:
$$
\|\lambda R_{i,\lambda }f\|_{L_q( \Sigma _i)}\le c\|f\|_{L_q(\Sigma
_i)},\quad
\| R_{i,\lambda }f\|_{\dot H_q^a(\ol \Sigma _i)}\le c_1\| R_{i,\lambda }f\|_{D_q(\ol \Sigma _i)}
\le c_2\|f\|_{L_q(\Sigma _i)},
$$
when $f\in L_q(\ol \Sigma _1)$.
Since  $(2a-1)_+\in \,[0,a)$, there is an
interpolation inequality (as in \eqref{eq:4.11})
\begin{equation}
  \label{eq:5.34}
\|v\|_{\dot H_q^{2a-1}(\ol \Sigma _i)} \le c_3\|v\|^\theta _{L_q(
\Sigma _i)} \|v\|^{1-\theta }_{\dot H_q^{a}(\ol \Sigma _i)},
\end{equation}
where $\theta =1-(2a-1)_+/a$, equal to $(1-a)/a$ if $a > \frac12$
and $1$ if $a\le \frac12$. Then for $u\in \dot
H_q^{(2a-1)_+}(\ol \Omega )$,
\begin{align*}
\| R_{i,\lambda }1_{\Sigma _i}(S_i\psi _i&+S_i')u\|_{\dot H_q^{(2a-1)_+}(\ol \Sigma _i)}\le c_3 \|
R_{i,\lambda }1_{\Sigma _i}(S_i\psi _i+S_i')u\|^\theta _{L_q( \Sigma _i)}\| R_{i,\lambda
}1_{\Sigma _i}(S_i\psi _i+S_i')u\|^{1-\theta }_{\dot H_q^{a}(\ol \Sigma _i)}\\
&\le c_3|\lambda |^{-\theta }(c_1\|1_{\Sigma _i}(S_i\psi _i+S_i')u\|_{L_q( \Sigma _i)})^\theta
(c_2\|1_{\Sigma _i}(S_i\psi _i+S_i')u\|_{L_q( \Sigma _i)})^{1-\theta}\\
&\le c_4|\lambda |^{-\theta }(\|\psi _iu\|_{\dot H_q^{(2a-1)_+}(\ol \Sigma
_i)}+\|u\|_{L_q(\Omega )})\le c_5|\lambda |^{-\theta }\|u\|_{\dot H_q^{(2a-1)_+}(\ol \Omega )}.
\end{align*}
It follows that 
$$
\|T_\lambda u\|_{\dot H_q^{(2a-1)_+}(\ol \Omega )}
=\Big\|\sum_{i=1}^N R_{i,\lambda }1_{\Sigma _i}(S_i\psi _i+S_i')u\Big\|_{\dot H_q^{(2a-1)_+}(\ol \Omega )}\le c_6|\lambda |^{-\theta }\|u\|_{\dot H_q^{(2a-1)_+}(\ol \Omega )}.
$$
Thus for $|\lambda |$ sufficiently large, $\sum_{k\ge 0}T_\lambda ^k$
converges in $\mathcal L(\dot H_q^{(2a-1)_+}(\ol \Omega ))$, so $1-T_\lambda $
is bijective there, and since $H_\lambda $ and $R_\lambda $ range in
the subspace $D_q(\comega)$, $R_\lambda $ identifies with $H_\lambda $.

Now let us show $\mathcal R$-boundedness for large $|\lambda |$. The $k$'th term
in the series is
$$
T_\lambda ^kR_{0,\lambda }=T_\lambda ^k\sum_{j=1}^N R_{j,\lambda }1_{\Sigma _j}\varrho _j.
$$
For $k=1$,   $\lambda T_\lambda R_{0,\lambda }=\lambda  \sum_{i,j=1}^NR_{i,\lambda }1_{\Sigma _i}(S_i\psi
_i+S_i')R_{j,\lambda }1_{\Sigma _j}\varrho _j$ has an $\mathcal R$-bound estimated by
\begin{align*}
&\mathcal R_{\mathcal L(L_q(\Omega ))}\Big\{\lambda \sum_{i,j=1}^NR_{i,\lambda }1_{\Sigma _i}(S_i\psi
                 _i+S_i')R_{j,\lambda }1_{\Sigma _j}\varrho _j\big|\lambda\in V_{\delta,K_1}\Big\}\\
  &\le \sum_{i,j=1}^N\mathcal R_{\mathcal L(L_q(\Omega ))}\{\lambda
R_{i,\lambda }|\lambda\in V_{\delta,K_1}\}\mathcal R_{\mathcal L(L_q(\Omega ))} \{1_{\Sigma _i}(S_i\psi
_i+S_i')R_{j,\lambda }1_{\Sigma _j}\varrho _j|\lambda\in V_{\delta,K_1}\}
\end{align*}
by the sum and product rules. Since $S_i\psi
_i+S_i'$ is of order $(2a-1)_+$, we can use Theorem~\ref{thm:4.7} $3^\circ$,
\eqref{eq:5.34} and the fact that $D_q(\ol\Sigma _1)\subset \dot H_q^a(\ol \Sigma _1)$ to
show that for $K_1\ge \ol K$,  the $\mathcal
R$-bound of the second factor in each term is $\le c K_1^{-\theta }$
when $\lambda \in V_{\delta ,K_1}$. Denote
$$
\max_{i\le N}\mathcal R_{\mathcal L(L_q(\Omega ))}\{\lambda
R_{i,\lambda }\mid \lambda \in V_{\delta ,\ol K}\}=C_0.
$$
For a given $0<\varepsilon <1$, take $K_1$ so large that for all
$i,j=1,\dots, N$,
\begin{equation}
  \label{eq:5.35}
\mathcal R_{\mathcal L(L_q(\Omega ))} \{1_{\Sigma _i}(S_i\psi
_i+S_i')R_{j,\lambda }1_{\Sigma _j}\varrho _j\mid \lambda \in V_{\delta
,K_1}\}\le \varepsilon .
\end{equation}
Then by summation over $i,j$,
$$
\mathcal R_{\mathcal L(L_q(\Omega ))}\{\lambda T_\lambda R_{0,\lambda }\mid \lambda \in V_{\delta
,K_1}\}\le C_0N^2\varepsilon .
$$
For $T^kR_{0,\lambda }$ there are similar formulas with $k$ factors of
the second type:
$$
T_\lambda ^kR_{0,\lambda }=\sum_{i_1,\dots, i_{k+1}=1}^NR_{i_1,\lambda }1_{\Sigma _{i_1}}(S_{i_1}\psi
_{i_1}+S_{i_1}')\dots R_{i_k,\lambda }1_{\Sigma _{i_k}}(S_{i_k}\psi
_{i_k}+S_{i_k}')
R_{i_{k+1},\lambda }1_{\Sigma _{i_{k+1}}}\varrho _{i_{k+1}}.
$$
Here we find  the estimate 
$$
\mathcal R_{\mathcal L(L_q(\Omega ))}\{\lambda T^k_\lambda R_{0,\lambda }\mid \lambda \in V_{\delta
,K_1}\}\le C_0N^{k+1}\varepsilon ^k.
$$
Then, if we adapt the choice of $K_1$ such that \eqref{eq:5.35} holds with $\varepsilon <1/N$, the
series \eqref{eq:5.33} converges with respect to $\mathcal R$-bounds (by \cite[Proposition~4.8]{DHP03}). Then $R_\lambda =H_\lambda $ has been determined and is $\mathcal
R$-sectorial on $V_{\delta ,K_1}$. 
\end{proof}

\begin{rem}\label{rem:5.9} It is seen from the proof that the evenness of
the symbol of $P$ is only needed in a small neighborhood of the
boundary; away from this,  strong ellipticity suffices.
\end{rem}

\section{Results for linear evolution equations}

We now turn to the consequences for heat problems.

Thanks to the results in Section~\ref{sec:5}, we can now obtain maximal
regularity results in much
more general cases than the one in Proposition~\ref{prop:5.2}.

\begin{thm}\label{thm:6.1}
    Let $\Omega $ be bounded with  $C^{1+\tau
}$-boundary for some $\tau >2a$, and let $1<p,q<\infty $. Let
$P=\operatorname{OP}(p)$ satisfy Hypothesis~\ref{hypo:3.1}, and
assume that the principal symbol $p_0(x,\xi )$ is real positive at
each boundary point   $x\in \partial\Omega$. Let $I=(0,T)$ for some
$T\in (0,\infty)$. 
Then for any
$f\in L_p(I;L_q(\Omega))$, the heat equation \eqref{eq:1.1} has  a unique solution $u\in
C^0(\overline{I};L_q(\Omega ))$  satisfying
\begin{equation}
  \label{eq:6.1}
u\in L_p(I;D_q(\overline \Omega ))\cap  H^1_p( I;L_q(\Omega
)).
\end{equation}
\end{thm}
\begin{proof} Because of Theorem~\ref{thm:5.8}, the shifted operator  $
P_{D,q}+k\colon D(P_{D,q})= H^{a(2a)}_q(\comega)\subset
L_q(\Omega)\to L_q(\Omega)$ satisfies the second statement of
Theorem~\ref{thm:4.4} for some $k>0$ sufficiently large. Hence $P_{D,q}+k$
has maximal $L_p$-regularity on $I=\R_+$. This implies that
$P_{D,q}$ has maximal $L_p$-regularity on $I=(0,T)$ for any $T\in
(0,\infty)$. 
\end{proof}

Note that the theorem allows $p\ne q$.

Nonhomogeneous boundary problems can also be considered. There is a
local  Dirichlet boundary condition associated with $P$, namely
the assignment of $\gamma _0(u/d_0^{a-1})$;
recall $d_0(x)=\operatorname{dist}(x,\partial\Omega )$ near
$\partial\Omega $, extended smoothly to $\Omega $. As shown in
earlier works (cf.\ \cite{G15}, \cite{G23}), it is natural to study the  problem
\begin{equation}
  \label{eq:6.2}
Pu=f\text{ in }\Omega ,\quad \gamma _0(u/d_0^{a-1})=\varphi ,\quad
\operatorname{supp}u\subset \comega, 
\end{equation}
for $u$ in the $(a-1)$-transmission space $H_q^{(a-1)(2a)}(\comega)$
(cf.\  \eqref{eq:2.14}ff.), which is mapped by $r^+P$ into $L_q(\Omega
)$ by \cite[Theorem~3.5]{G23}. This is a larger space
than $D_q(\comega)=H_q^{a(2a)}(\comega)$, satisfying
\begin{equation}
  \label{eq:6.1a}
H_q^{a(2a)}(\comega)=\{u\in H_q^{(a-1)(2a)}(\comega)\mid \gamma
_0(u/d_0^{a-1})=0\}.
\end{equation}
The problem \eqref{eq:6.2} is Fredholm solvable with $u\in H_q^{(a-1)(2a)}(\comega)$
 for $f,\varphi $ given in
$L_q(\Omega )$ resp.\ $B_{q,q}^{a+1-1/q}(\partial\Omega )$, when $\tau
>2a+1$ \cite[Theorem~5.1]{G23}.

Note that the case $\varphi =0$ in \eqref{eq:6.2} is the homogeneous Dirichlet
problem. There is the notation for the boundary mapping, provided with a
normalizing constant, $$
\gamma _0^{a-1}\colon u\mapsto \Gamma (a+1)\gamma
_0(u/d_0^{a-1}).$$

By \cite[Theorem~2.3]{G23} with $\mu =a-1$, there holds:

\begin{prop}\label{prop:6.2a}
   When $\tau \ge 1$ and $a-1+\frac1q<s<\tau
$ with $s<\tau +a-1$,
the mapping $\gamma _0^{a-1}$ is continuous from
$H_q^{(a-1)(s)}(\comega)$ to $
B_{q,q}^{s-a+1-\frac1q}(\partial\Omega )$ and has a right inverse $K_{(0)}^{a-1}$ that
maps continuously
$$
K_{(0)}^{a-1}\colon B_{q,q}^{s-a+1-\frac1q}(\partial\Omega )\to
H_q^{(a-1)(s)}(\comega).  
$$  
\end{prop}

In particular,
\begin{equation}
  \label{eq:6.2a}
K^{a-1}_{(0)}\colon B_{q,q}^{a+1-\frac1q}(\partial\Omega )\to
H_q^{(a-1)(2a)}(\comega),\;K^{a-1}_{(0)}\colon B_{q,q}^{\varepsilon
}(\partial\Omega )\to H_q^{(a-1)(a-1+\frac1q+\varepsilon
)}(\comega), 
\end{equation}
for  $\varepsilon >0$ (subject to $s=a-1+\frac1q+\varepsilon <\tau +a-1$).

By Lemma~5.3 in \cite{G23}, $H_q^{(a-1)(s)}(\comega)\subset L_q(\Omega )
$ for $s\ge 0$, when $q<\frac1{1-a}$. We assume this for the nonhomogeneous heat
problem:
\begin{equation}
  \label{eq:6.3}
  \begin{split}
    \partial_tu+Pu&=f\text{ on }\Omega \times I ,\\
     \gamma _0(u/d_0^{a-1})&=\psi \text{ on }\partial\Omega \times I ,\\
u&=0\text{ on }(\R^n\setminus\Omega )\times I, \\
u|_{t=0}&=0.
    \end{split}
\end{equation}

   Here we can show:
   
   \begin{thm}\label{thm:6.2}
          In addition to the assumptions of Theorem~\ref{thm:6.1},
 assume  that $\tau >2a+1$ and  $q<\frac1{1-a}$. Then \eqref{eq:6.3} 
has for  $f\in
L_p(I;L_q(\Omega ))$, $\psi \in L_p(I;B_{q,q}^{a+1-1/q}(\partial\Omega ))\cap  
H_p^1(  I; B_{q,q}^\varepsilon (\partial\Omega ))$ with $\psi
 (x,0)=0$ ($\varepsilon >0$) a unique
solution $u$ satisfying
\begin{equation}
  \label{eq:6.4b}
u\in L_p(I;H_q^{(a-1)(2a)}(\comega ))\cap   H_p^1(I; L_q(\Omega )).
\end{equation}
\end{thm}
 
\begin{proof} Considering the boundary mapping and its right inverse as
constant in $t$, we can add a time-parameter $t$, and have in view of
Propostion~\ref{prop:6.2a} and \eqref{eq:6.2a} for any $p\in (1,\infty )$ that with $I=(0,T)$,
\begin{align*}
\gamma _0^{a-1}&\colon L_p(I; H_q^{(a-1)(2a)}(\comega))\to L_p(I; B_{q,q}^{a+1-\frac1q}(\partial\Omega )),\\
\gamma _0^{a-1}&\colon   H^1_p(I; H_q^{(a-1)(a-1+\frac1q+\varepsilon
)}(\comega))\to
  H^1_p(I; B_{q,q}^{\varepsilon }(\partial\Omega )),
\end{align*}
with right inverses $K^{a-1}_{(0)}$ continuous in the opposite direction.

For the given $\psi$ as in the assumptions, let $v(x,t)=K^{a-1}_{(0)}\psi (x,t)$; it
lies in $L_p(I;H_q^{(a-1)(2a)}(\comega ))$ and in $   H^1_p(I; L_q(\Omega ))$ (since
$H_q^{(a-1)(a-1+\frac1q +\varepsilon )}(\comega )\subset L_q(\Omega )$), 
and satisfies 
$$
\gamma _0^{a-1}v=\psi ,\quad v|_{t=0}=0,\quad r^+Pv \in L_p(I;
L_q(\Omega )),\quad \partial_tv\in L_p(I;L_q(\Omega )).
$$
Then $w=u-v$ is in $L_p(I; H_q^{(a-1)(2a)}(\comega))$ with $\gamma
_0^{a-1}w=0$, hence in $L_p(I; H_q^{a(2a)}(\comega))$ by \eqref{eq:6.1a}.
Moreover, $(r^+P+\partial_t)(u-v) \in L_p(I;L_q(\Omega ))$.
Thus in order for $u$ to solve \eqref{eq:6.3}, $w$ must 
solve a problem \eqref{eq:1.1} with homogeneous boundary
condition and $f$ replaced by $f-(r^+P+\partial_t)v$. Here Theorem 6.1 assures that there is a unique
solution $w\in L_p(I;H_q^{a(2a)}(\comega ))\cap   H_p^1(I; L_q(\Omega
))$. Then $u=v+w$ is the unique solution of \eqref{eq:6.3}, satisfying \eqref{eq:6.4b}.  
\end{proof}

Let us also mention that one can  use the resolvent estimates
(just in uniform norms) to show results for other function spaces. For
example,  by a strategy of Amann \cite{A97}:

\begin{thm}\label{thm:6.3}
  Assumptions as in Theorem~\ref{thm:6.1}.  Let
   $s$ be noninteger $>0$. For any $f\in  \dot C^s
   (\overline{\R}_+ ;L_q(\Omega ))$ there is a unique solution
 $u\in \dot C^s (\overline{\R}_+
     ;D_q(\overline\Omega))$, and there holds
     \begin{equation}
       \label{eq:6.4}
f\in \dot C^s(\overline{\R}_+ ;L_q(\Omega ))\iff 
u\in \dot C^s(\overline{\R}_+
;D_q(\overline\Omega))\cap \dot C^{s+1} (\overline{\R}_+ ;L_q(\Omega )).
     \end{equation}
\end{thm}
 \begin{proof} The proof goes exactly as in
   \cite[Theorem~5.14]{G18b}. The notation $\dot C^s(\overline{\R}_+
   ;X)$ indicates the functions in $ C^s({\R} ;X)$ vanishing for $t<0$. 
\end{proof}

As in Remark~\ref{rem:5.9} we observe that the evenness of the symbol $p(x,\xi
)$ is only needed in a small neighborhood of the boundary.

\section{Applications to nonlinear evolution equations} 

In this last section we present an application of the result on
maximal regularity established in Theorem~\ref{thm:6.1} to existence of strong solutions of the nonlinear nonlocal parabolic equation
\begin{equation}\label{eq:NonEq}
  \begin{aligned}
  \partial_tu+ a_0(x,u)Pu &= f(x,u)&\quad& \text{in }\Omega\times (0,T),\\
  u&=0 &\quad&\text{on } (\R^n\setminus \Omega)\times (0,T),\\
  u|_{t=0} &= u_1 &\quad& \text{in }\Omega,
  \end{aligned}
\end{equation}
for some $T>0$. 
\begin{thm}\label{thm:Nonlinear}
   Let $\Omega $ be a bounded domain with  $C^{1+\tau
   }$-boundary for some $\tau >2a$, and let $1<p,q<\infty $ be such that
   \begin{equation}
     \label{eq:Ineq1}
     (a+\tfrac1q)(1-\tfrac{1}p) -\tfrac{n}q >0.
   \end{equation}
If $n=1$, assume moreover $\frac1q<a$. Let
$P$ satisfy Hypothesis~\ref{hypo:3.1}, and
assume that the principal symbol $p_0(x,\xi )$ is real positive at
each boundary point   $x\in \partial\Omega$. Moreover, for an open set 
$U\subset \R$  with $0\in U$,  let $a_0\in
C^{\max(1,\tau)}(\R^n\times U,\R)$ with $a_0(x,s)>0$ for all $s\in U$ and
$x\in\R^n$, let $f\colon \Rn\times U\to \R\colon (x,u)\mapsto f(x,u)$ be continuous and locally Lipschitz with respect to $u\in U$, and let
$u_0 \in (L_q(\Omega),D_q(\overline\Omega))_{1-\frac1p,p}\cap
C^\tau(\overline{\Omega}) $ with $\overline{u_0(\Omega)}\subset
U$. Then there are $\eps_0,T>0$ such that for every
$u_1\in
X_{\gamma,1}:=(L_q(\Omega),D_q(\overline\Omega))_{1-\frac1p,p}$ with
$\|u_0-u_1\|_{X_\gamma,1}\leq \eps_0$, the system
\eqref{eq:NonEq} possesses a unique solution
\begin{equation*}
  u\in L_p((0,T);D_q(\overline{\Omega}))\cap H^1_p((0,T);L_q(\Omega)).
\end{equation*}
\end{thm}

\begin{proof}
  We prove the result by applying a local existence result for an
  abstract evolution equation by K\"ohne et
  al.~\cite[Theorem~2.1]{KPW10}, which can also be found in
  \cite[Theorem~5.1.1]{PS16}. Alternatively, one could also use a
  result by Cl\'ement and Li~\cite[Theorem~2.1]{CL93}. To this end we
  choose $X_0 = L_q(\Omega)$, $X_1=D_q(\overline\Omega)$. Note that
  \eqref{eq:Ineq1} implies $\frac 1q<a$ when $n\ge 2$, so that
  $D_q(\overline\Omega)\hookrightarrow
  \dot{H}_q^{a+\frac1q-\eps}(\comega)$ by \eqref{eq:3.5a} for all
  $n\ge 1$. Here $\dot{H}_q^{a+\frac1q-\eps}(\comega)\hookrightarrow
  \ol{H}_q^{a+\frac1q-\eps}(\Omega)$.
  
  Then in the notation of \cite{KPW10} (with $\mu=1$)
  \begin{align}   \nonumber
    X_{\gamma,1}&:=(L_q(\Omega),D_q(\overline\Omega))_{1-\frac1p,p}\\\label{eq:Embedding}
    &\hookrightarrow (L_q(\Omega),\ol{H}_q^{a+\frac1q-\eps}(\Omega))_{1-\frac1p,p}= \ol{B}_{q,p}^{(a+\frac1q-\eps)(1-\frac1p)}(\Omega) 
    \hookrightarrow C^0(\ol{\Omega})  
  \end{align}
  for $\eps>0$ sufficiently small, in view of \eqref{eq:Ineq1}.
 Moreover, let
  \begin{equation*}
     V_1:= \{u \in X_{\gamma,1}\mid  u(x) \in U \text{ for all }x\in \ol{\Omega}\}.
  \end{equation*}
  Then $V_1\subset X_{\gamma,1}$ is open due to \eqref{eq:Embedding}
  and the fact that $U\subset \R$ is open. Moreover, since $a_0, f\colon U\to \R$ are locally Lipschitz continuous, we have that
  $$
  u\mapsto a_0(\cdot,u(\cdot) ), u\mapsto f(\cdot, u(\cdot)) \in C^{0,1} (V_1, C^0(\overline\Omega)).
  $$ Now we define $A\colon V_1\to \mathcal{L}(X_1,X_0)$ and $F\colon V_1  \to X_0$ by
  \begin{equation*}
    A(u)= a_0(\cdot, u(\cdot)) P,\quad F(u)= f(\cdot,u(\cdot)) \qquad \text{for all }u\in V_1.
  \end{equation*}
  Because of $P\in \mathcal{L}(D_q(\overline{\Omega}),L_q(\Omega))$, this yields
  \begin{align*}
    A\in C^{0,1} (V_1, \mathcal{L}(X_1,X_0)),\ F\in C^{0,1} (V_1, X_0). 
  \end{align*}
  Finally, we note that, since $u_0 \in X_{\gamma,1} \cap C^\tau(\overline{\Omega})$, we have $a_0(\cdot, u_0(\cdot)) \in C^\tau(\overline{\Omega})$. Thus $A(u_0)=\Op(\tilde{p})$,
with $\tilde{p}(x,\xi)= a_0(x,u_0(x))p(x,\xi)$ for all $x,\xi\in\R^n$, satisfies again Hypothesis~\ref{hypo:3.1}. Therefore $A(u_0)$ has maximal $L_p$-regularity on every finite time interval $I=(0,T)$, $0<T<\infty$ due to Theorem~\ref{thm:6.1}. Hence all assumptions of \cite[Theorem~2.1]{KPW10} with $\mu=1$ are satisfied. This yields the statement of the theorem.
\end{proof}
\begin{rem}\label{rem:Uniqueness}
  Actually, the uniqueness statement in Theorem~\ref{thm:Nonlinear} holds in a slightly stronger local sense: If $u,\tilde{u}\in L_p((0,T');D_q(\overline{\Omega}))\cap H^1_p((0,T');L_q(\Omega))$ are solutions of \eqref{eq:NonEq} with $(0,T)$ replaced by $(0,T')$ for some $T'\in (0,T]$ and initial value as before, then $u\equiv \tilde{u}$. This follows immediately from the proof of \cite[Theorem~2.1]{KPW10}, which is based on the contraction mapping principle and uses that $T$ is sufficiently small.
\end{rem}

Finally, we apply the previous result to a fractional nonlinear diffusion equation with a nonzero exterior condition, of the form
\begin{equation}\label{eq:NonlinearDiffusion}
  \begin{aligned}
  \partial_tw+ P\varphi(w) &= 0&\quad &\text{in }\Omega\times (0,T),\\
  w&=w_b &&\text{on } (\R^n\setminus \Omega)\times (0,T),\\
  w|_{t=0} &= w_1 &&\text{in }\Omega,
  \end{aligned}
\end{equation}
for some function
$\varphi \in  {C^1(\overline{\rp},\R)\cap} C^2(\rp,\R)$ with  {$\varphi(0)=0$} and
$\varphi'(s)>0$ for all $s\in \rp$.
\begin{cor}\label{cor:Diffusion}
   Let $\Omega $ be bounded with  $C^{1+\tau
}$-boundary for some $\tau >2a$, and let $1<p,q<\infty $ be such that
\eqref{eq:Ineq1} holds, assuming also $\frac1q<a$ if
$n=1$. Let
$P$ satisfy Hypothesis~\ref{hypo:3.1}, and
assume that the principal symbol $p_0(x,\xi )$ is real positive at
each boundary point   $x\in \partial\Omega$, and that $P$ maps real
functions to real functions.
Moreover, let
$\varphi \in C^2(\rp)$ be real with  $\varphi'(s)>0$ for all
$s\in\rp$, let $w_b\in {H}^{2a}_q(\Rn)\cap C^\tau(\Rn) $ be
real with  {$\inf_{x\in\overline{\Omega}}w_b(x)>0$}, and let
$w_0\colon \ol{\Omega}\to \rp$ be such that
$\varphi(w_0)-\varphi(w_b)\in
(L_q(\Omega),D_q(\overline\Omega))_{1-\frac1p,p}\cap
C^\tau(\overline{\Omega})$. Then there is some $\eps_0>0$ such that
for every $w_1\colon \Omega\to \rp$ with
$\varphi(w_1)-\varphi(w_0)\in
X_{\gamma,1}$ (cf.\ \eqref{eq:Embedding})
and
$\|\varphi(w_0)-\varphi(w_1)\|_{X_{\gamma,1}}\leq \eps_0$, the system
\eqref{eq:NonlinearDiffusion}
possesses a unique solution $w\in \bigcap_{0\leq s
  <\frac1q}L_p((0,T);\ol{H}^{a+s}_q(\Omega))\cap H^1_p((0,T) ;L_q(\Omega))$ with $\varphi(w)-\varphi(w_b)\in L_p((0,T);D_q(\overline{\Omega}))$ for some $T>0$.
\end{cor}
\begin{proof}
  We use a reformulation of
  \eqref{eq:NonlinearDiffusion} in the
  form \eqref{eq:NonEq}. First of all, in view of
  \eqref{eq:Ineq1}ff.,
    we have $w_0 \in C^0(\ol{\Omega})$ and therefore
  \begin{equation*}
  \delta:= \min\big\{ \inf_{x\in\ol\Omega} w_0(x), \inf_{x\in {\overline{\Omega}}}w_b(x)\big\}>0.  
  \end{equation*}
  Hence there is some $\tilde{\varphi}\in C^2(\R,\R)$ with $\tilde{\varphi}'(s)>0$ for all $s\in\R$ and $\tilde{\varphi}(s)=\varphi(s)$ for all $s\geq \frac{\delta}2$.  {Furthermore, we choose some $\tilde{w}_b\in C^\tau(\R^n)$ such that $\tilde{w}_b|_{\Omega}= w_b|_{\Omega}$ and $\inf_{x\in \R^n}\tilde{w}_b(x)>0$.} Moreover, we define
  \begin{alignat*}{2}
    a_0(x,s)&= \tilde\varphi'(\tilde{\varphi}^{-1}(s+\varphi( {\tilde{w}_b}(x))))&\quad& \text{for all }s\in U:=\R, x\in\R^n,\\
    f(x,s)&=-a_0(x,s)P(\varphi(w_b(x)))(x)&\quad&\text{for all }s\in\R, x\in \Omega,
  \end{alignat*}
  and $u_0:=
  \tilde{\varphi}(w_0)-\tilde{\varphi}(w_b)=\varphi(w_0)-\varphi(w_b)$. Hence
  we can apply Theorem~\ref{thm:Nonlinear} and get the existence of
  some $\eps_0>0$ and $T>0$ such that for every $w_1\colon \Omega\to
  \rp$ with $\tilde\varphi(w_1)-\varphi(w_0)\in X_{\gamma,1}$ (cf.\
  \eqref{eq:Embedding})
    and $\|\tilde{\varphi}(w_0)-\varphi(w_1)\|_{X_{\gamma,1}}\leq \eps_0$ there is a unique solution $u\in L_p((0,T);D_q(\overline{\Omega}))\cap H^1_p((0,T);L_q(\Omega))$ of \eqref{eq:NonEq}.
  Moreover, by choosing $\eps_0>0$ sufficiently small, we can achieve
  that $\|\tilde{\varphi}(w_0)-\varphi(w_1)\|_{X_{\gamma,1}}\leq
  \eps_0$ implies $\|w_0-w_1\|_{C^0(\ol{\Omega})}<\delta/2$, since
  $\tilde{\varphi}\colon \R\to \R$ is strictly monotone. Hence
  $\inf_{x\in \ol\Omega}w_0>\delta/2$ and $\tilde{\varphi}(w_0)=
  \varphi(w_0)$ in that case. Now let us define $w:=
  \tilde{\varphi}^{-1}(u+\varphi(w_b))$.

  Then $w\in  L_p((0,T);\ol{H}^{a+\frac1q-\eps}_q(\Omega))\cap H^1_p((0,T);L_q(\Omega))$ since
  $$
  u+\varphi(w_b)\in L_p((0,T);\ol{H}^{a+\frac1q-\eps}_q(\Omega))\cap H^1_p((0,T);L_q(\Omega)),
  $$
  $\tilde{\varphi}\in C^2(\R)$ and by well-known results on composition operators on Sobolev and Bessel potential spaces and
  \begin{align*}
    \partial_t w&= (\tilde{\varphi}^{-1})'(u+\varphi(w_b))\partial_t u = -(\tilde{\varphi}^{-1})'(u+\varphi(w_b))a_0(\cdot, u(\cdot))P(u+\varphi(w_b))= -P(\tilde{\varphi}(w)).
  \end{align*}
  Moreover, since
  \begin{align*}
    &L_p((0,T);\ol{H}^{a+\frac1q-\eps}_q(\Omega))\cap
      H^1_p((0,T);L_q(\Omega))\\
      &\hookrightarrow BUC([0,T];\ol{H}^{(a+\frac1q-\eps)(1-\frac1p)}_q(\Omega))\hookrightarrow C^0([0,T]\times \ol\Omega)
  \end{align*}
  for $\eps>0$ sufficiently small due to \eqref{eq:Ineq1}ff.\
    and $\inf_{x\in \ol\Omega}w_0>\delta$, we can achieve
  $$
  \inf_{x\in \ol\Omega,t\in [0,T]}w(x,t)> \delta/2
  $$
  by choosing $T>0$ sufficiently small. Hence $\tilde{\varphi}(w)= \varphi(w)$. Finally, $\tilde{\varphi}(w)-\varphi(w_b)= u \in L_p((0,T);D_q(\overline{\Omega}))$ by definition. This shows existence of a solution.

  It remains to show uniqueness of the constructed solution $w$. To this end let $\tilde{w}\in L_p((0,T);\ol{H}^{a+\frac1q-\eps}_q(\Omega))\cap H^1_p((0,T);L_q(\Omega))$ with $\varphi(\tilde{w})-\varphi(w_b)\in L_p((0,T);D_q(\overline{\Omega}))$ be another solution of \eqref{eq:NonlinearDiffusion} and consider
  \begin{equation*}
    t_0 := \sup \left\{ T'\in [0,T]\mid  w|_{[0,T']}\equiv \tilde{w}|_{[0,T']}\right\}.
  \end{equation*}
  We show by contradiction that $t_0=T$, which implies the uniqueness. Hence assume $t_0<T$. Since $w,\tilde{w}\in C^0([0,T]\times \ol\Omega)$, we have
  \begin{equation*}
    \inf_{x\in\ol\Omega} \tilde{w}(x,t_0)=\inf_{x\in\ol\Omega} w(x,t_0)> \delta/2.
  \end{equation*}
  Hence there is some $T'\in (t_0,T)$ such that
  \begin{equation*}
    \inf_{x\in\ol\Omega,t\in [t_0,T']} \tilde{w}(x,t)> \delta.
  \end{equation*}
  Therefore $\tilde{u}:= \varphi(\tilde{w})|_{[0,T']}= \tilde{\varphi}(\tilde{w})|_{[0,T']}\in L_p((0,T');D_q(\overline{\Omega}))\cap H^1_p((0,T');L_q(\Omega))$ is a solution of \eqref{eq:NonEq} with $(0,T)$ replaced by $(0,T')$. Since $u|_{[0,T']}$ solves the same system, the improved uniqueness statement of Remark~\ref{rem:Uniqueness} implies that $\tilde{u}|_{[0,T']}= u|_{0,T']}$. This yields $\tilde{w}|_{[0,T']}=w|_{[0,T']}$, which is a contradiction to the definition of $t_0$. Hence $t_0=T$, and uniqueness is shown.
\end{proof}

\begin{example}\label{ex:Porous} 
Choosing $\varphi(w)= w^m$ for $m>1$ in
\eqref{eq:NonlinearDiffusion} yields
a case including
the fractional porous medium equation; in the latter, $P\varphi (w)=(-\Delta )^aw^m$.

The problem with $\varphi $ was  studied e.g.\ in H\"older spaces in the case $\Omega=\R^n$
 and $P=(-\Delta)^a$
by   V\'azques, de~Pablo, Quir\'{o}s and Rodr\'{\i}guez in
   \cite{VPQR17}, which lists  a number of applications including the
   fractional porous medium equation.  Roidos and
Shao obtained maximal $L_p$-regularity results in \cite{RS22} in
cases like $P=(-\nabla\cdot \mathfrak a(x)\nabla)^a$, with $\Omega$  replaced by
a smooth closed $n$-dimensional Riemannian manifold; they applied it
in their Section 6.1
to porous medium equations for $P=(-\Delta )^a$. The present
study achieves these types of results for the first time on domains
$\Omega $
{\em with
boundary}; examples  include $P=L^a$ where $L$ is as in
\eqref{eq:diffop} below.
\end{example}

Corollary \ref{cor:Diffusion} applies moreover to pseudodifferential
operators $P$ satisfying Hypothesis~\ref{hypo:3.1} with  $p(x,\xi
)$  real and vanishing odd-numbered symbol terms $p_{2k+1}, k\in\N_0$,
so that $p(x,-\xi )=p(x,\xi )$; cf.\ Remark \ref{rem:RealMap} below.

For completeness, we give some details on when operators in complex
function spaces map real functions to real functions:

\begin{rem}\label{rem:RealMap}  A
function $u\in\mathcal S(\rn)$ is real if and only if $\hat u(-\xi
)=\overline{\hat u(\xi )}$ for all $\xi \in\rn$. It follows from \eqref{eq:2.8}
that $P=\Op(p)$ maps real functions to real functions if and only if 
$p(x,-\xi) = \overline{p(x,\xi)}$  for all $x,\xi\in \rn$. This gives
one criterion for preserving real functions.

For operators arising
from functional calculus, another criterion may be convenient: When
$A$ is a linear operator in $L_q(\rn ,\C)$  {with $\overline{u}\in D(A)$ for every $u\in D(A)$}, define $\overline A $
by $\overline Au=\overline{A\overline u}$, with $D(\overline
A)=D(A)$. Then $A$ maps real functions to real functions if and only
if $\overline A=A$. Assume this, and let $f$ be  a function on $\C$, holomorphic on
the resolvent set of $A$, satisfying
\begin{equation}
  \label{eq:real}
\overline{f(\lambda )}=f(\overline \lambda ).
\end{equation}
Let the operator
$f(A)$ be defined by a Dunford integral
$f(A)u=\tfrac i{2\pi }\int_{\mathcal C}f(\lambda )(A-\lambda )^{-1}u\, d\lambda$, 
where $\mathcal C$ is a curve encircling the spectrum of $A$ counterclockwise.
Note that from $\overline{(A-\lambda )}=A-\overline\lambda $ follows
$\overline{(A-\lambda )^{-1}}=(A-\overline\lambda )^{-1}$ when $\lambda
$ is in the resolvent set. Hence, in view of \eqref{eq:real},
\begin{align*}
\overline{f(A)}u&=\overline{\tfrac i{2\pi }\int_{\mathcal C}f(\lambda
)(A-\lambda )^{-1}\overline u\, d\lambda }
=\tfrac {-i}{2\pi }\int_{\mathcal C}f(\overline\lambda
)(\overline A-\overline\lambda )^{-1} u\, d \overline\lambda \\
&=\tfrac {i}{2\pi }\int_{\mathcal C'}f(\mu )(\overline A-\mu  )^{-1} u\, d\mu =f(A)u,
\end{align*}
 since
$\overline A=A$ (here $\mathcal C'$ is the curve obtained by conjugation
of $\mathcal C$ and oriented counterclockwise). Thus $f(A)$ preserves real functions.

This can be used for example when  $P=L^a$, where 
\begin{equation}
  \label{eq:diffop}
 Lu= -\sum_{j,k=1}^{n}\partial_ja_{jk}(x)\partial_ku+b(x)u,
\end{equation}
with $(a_{jk}(x))_{1\leq j,k\le n}$ being a real, symmetric,
 $x$-dependent matrix with a positive lower bound for $x\in\rn$, and $b(x)\ge
 0$. $L$ preserves real functions. The fractional
powers $L^a=LL^{a-1}$ ($0<a<1$) can be defined under mild smoothness hypotheses on the coefficients;
then they also preserve real functions. When
all coefficients are in $C^\infty _b(\rn)$, the construction of Seeley \cite{S67}
shows that $P$ has a smooth symbol satisfying Hypothesis \ref{hypo:3.1}. When
coefficients are just $C^\tau $, there is a principal symbol $p_0=\bigl(\sum
a_{jk}(x)\xi _j\xi _k\bigr)^a$ satisfying Hypothesis \ref{hypo:3.1}
, but the remainder term  would need further analysis.
\end{rem}

{\bf Data availability statement} There is no associated data to the
manuscript.

\medskip

{\bf Conflict of interest} There is no conflict of interests. 
\small

\def\cprime{$'$} \def\ocirc#1{\ifmmode\setbox0=\hbox{$#1$}\dimen0=\ht0
  \advance\dimen0 by1pt\rlap{\hbox to\wd0{\hss\raise\dimen0
  \hbox{\hskip.2em$\scriptscriptstyle\circ$}\hss}}#1\else {\accent"17 #1}\fi}

\end{document}